\documentclass[11pt]{article}
\usepackage[margin=1in]{geometry}
\usepackage{amsfonts}
\usepackage{graphicx}
\usepackage{epstopdf}
\usepackage{algorithmic}
\usepackage{hyperref}
\usepackage{commath}
\usepackage{amsmath,amsthm}

\newtheorem{theorem}{Theorem}
\newtheorem{lemma}{Lemma}

\newtheorem{corollary}{Corollary}

\newcommand{\calX}{\mathcal{X}}

\newcommand{\Gcal}{\mathcal{G}}
\newcommand{\Vcal}{\mathcal{V}}
\newcommand{\Ecal}{\mathcal{E}}

\newcommand{\DX}{\mathcal{D}_\mathcal{X}}
\newcommand{\bfone}{\mathbf{1}}
\newcommand{\taubar}{\bar{\tau}}
\DeclareMathOperator{\Ex}{\mathbb{E}}
\DeclareMathOperator{\prob}{\mathrm{P}}
\DeclareMathOperator{\proj}{\Pi}
\def\argmin{\mathop{\rm argmin}}
\def\minimize{\mathop{\rm minimize}}

\title{Decentralized Consensus Algorithm with Delayed and Stochastic Gradients}

\author{
Benjamin Sirb \and Xiaojing Ye\thanks{Department of Mathematics \& Statistics,
Georgia State University, Atlanta, GA 30303, USA
(\url{bsirb1@student.gsu.edu}, \url{xye@gsu.edu}).
This work was partially supported by National Science Foundation under grants
DMS-1620342 and CMMI-1745382.}}
\date{}

\usepackage{amsopn}

\begin{document}
\maketitle

\begin{abstract}
We analyze the convergence of decentralized consensus
algorithm with delayed gradient information across the network. 
The nodes in the network privately hold parts of the objective 
function and collaboratively solve for the consensus
optimal solution
of the total objective while they can only communicate with their immediate neighbors.
In real-world networks, it is often difficult and sometimes impossible to 
synchronize the nodes, and therefore they have to use stale gradient information 
during computations.
We show that, as long as the random delays are bounded in expectation
and a proper diminishing step size policy is employed,
the iterates generated by decentralized gradient descent method 
converge to a consensual optimal solution.
Convergence rates of both objective and consensus are derived.
Numerical results on a number of synthetic problems and real-world seismic tomography datasets
in decentralized sensor networks are presented to show the performance of the method.
\bigskip

\bigskip

\end{abstract}

%

\section{Introduction}
In this paper, we consider a decentralized consensus optimization problem arising
from emerging technologies such as distributed machine learning 
\cite{Boyd:2011a,Forero:2010a,Kraska:2013a,Li:2014a}, 
sensor network \cite{Iutzeler:2012a,Rabbat:2004a,Song:2009a}, 
and smart grid \cite{Gan:2013a,Lo:2013a}. Let $\mathcal{G}=(\mathcal{V},\mathcal{E})$
be a network (undirected graph) where
$\mathcal{V}=\{1,2,\dots,m\}$ is the node (also called agent, processor, or sensor) 
set and $\mathcal{E}\subset \mathcal{V} \times \mathcal{V}$ is the edge set. Two nodes $i$ and $j$
are called neighbors if $(i,j)\in \mathcal{E}$. The communications between neighbor nodes are 
bidirectional, meaning that $i$ and $j$ can communicate with each other as long
as $(i,j)\in \mathcal{E}$.

In a decentralized sensor network $\mathcal{G}$, individual nodes can acquire, store, and process data
about large-sized objects.
Each node $i$ collects data and holds objective function $F_i(x;\xi_i)$ privately where $\xi_i\in\Theta$
is random with fixed but unknown probability distribution in domain $\Theta$
to model environmental fluctuations such as noise in data acquisition and/or inaccurate estimation of objective function or its gradient. 
Here $x\in X$ is the unknown (e.g., the seismic image) to be solved,
where the domain $X\subset\mathbb{R}^n$ is compact and convex. 
Furthermore, we assume that $F_i(\cdot;\xi_i)$ is convex for all $\xi_i\in\Theta$ and
$i\in \mathcal{V}$, and we define $f_i(x)=\Ex_{\xi_i}[F_i(x;\xi_i)]$ which is thus convex with respect to $x\in X$.
The goal of decentralized consensus optimization is to solve the minimization problem
\begin{equation}\label{eqn:deccon}
\minimize_{x \in X} f(x),\quad \mbox{where} \ f(x):=\sum_{i=1}^m f_i(x)
\end{equation}
with the restrictions that $F_i(x;\xi_i)$, and hence $f_i(x)$, are accessible by node $i$ only,
and that nodes $i$ and $j$ can communicate only if $(i,j)\in \mathcal{E}$
during the entire computation.

There are a number of practical issues that need to be taken into consideration 
in solving the real-world decentralized consensus optimization problem \eqref{eqn:deccon}:
\begin{itemize}
\item The partial objective $F_i$ (and $f_i$) is held privately by
node $i$, and transferring $F_i$ to a data fusion center is either infeasible or 
cost-ineffective due to data privacy, the large size of $F_i$, and/or 
limited bandwidth and communication power overhead of sensors. 
Therefore, the nodes can only communicate their own estimates of $x\in\mathbb{R}^n$ with their neighbors
in each iteration of a decentralized consensus algorithm.

\item Since it is often difficult and sometimes impossible for the nodes to be
fully synchronized, they may not have access to the most up-to-date
(stochastic) gradient information during computations. In this case, the
node $i$ has to use out-of-date (stochastic) gradient $\nabla F_i(x_i(t-\tau_i(t));\xi_i(t-\tau_i(t)))$
where $x_i(t)$ is the estimate of $x$ obtained by node $i$ at iteration $t$,
and $\tau_i(t)$ is the level of (possibly random) delay of the gradient information at $t$.

\item The estimates $\{x_i(t)\}$ by the nodes should tend to be consensual as $t$ increases,
and the consensual value is a solution of problem \eqref{eqn:deccon}. In this case, there is
a guarantee of retrieving a good estimate of $x$ from any surviving node in the network 
even if some nodes are sabotaged, lost, or run out of power during the computation process.
\end{itemize}

In this paper, we analyze a decentralized consensus algorithm which takes all the factors above
into consideration in solving \eqref{eqn:deccon}. We provide comprehensive convergence
analysis of the algorithm, including the decay rates of objective function and disagreements
between nodes, in terms of iteration number, level of delays, and network structure etc.

\subsection{Related work}


Distributed computing on networks is an emerging technology with extensive applications in 
modern machine learning \cite{Forero:2010a,Kraska:2013a,Li:2014a},
sensor networks \cite{Iutzeler:2012a,Rabbat:2004a,Zhao:2015b,Zhao:2015a}, and big data analysis \cite{Cevher:2014b,Sayed:2014a}.
There are two types of scenarios in distributed computing: centralized and decentralized.
In the centralized scenario, computations are carried out locally by worker (slave) nodes while 
computations of certain global variables must eventually be processed by designated master node 
or at a center of shared memory during each (outer) iteration. 
A major effort in this scenario has been devoted to update the global variable more effectively
using an asynchronous setting in, for example, distributed centralized alternating direction method
of multipliers (ADMM) \cite{Chang:2016a,Chang:2016b,Liu:2015a,Wei:2013b,Zhang:2014b}.
In the decentralized scenario considered in this paper, the nodes 
privately hold parts of objective functions and can only communicate with neighbor
nodes during computations. In many real-world applications, decentralized computing is particularly useful when 
a master-worker network setting is either infeasible or not economical, or the data acquisition
and computation have to be carried out by individual nodes which then need to collaboratively
solve the optimization problem. Decentralized networks are also more robust to node 
failure and can better address privacy concerns.
For more discussions about motivations and advantages of decentralized computing,
see, e.g., \cite{Jakovetic:2014a,Nedic:2009b,Olfati-Saber:2004a,Shi:2015a,Tian:2008a,Tsitsiklis:1984a} and
references therein.


Decentralized consensus algorithms take the data distribution and communication restriction 
into consideration, so that they can be implemented at individual nodes in the network. 
In the \textit{ideal synchronous case} of decentralized consensus where all the nodes
are coordinated to finish computation and 
then start to exchange information with neighbors in each iteration, a number of 
developments have been made. A class of methods is to rewrite the consensus constraints for
minimization problem \eqref{eqn:deccon} by introducing auxiliary variables between
neighbor nodes (i.e., edges),
and apply ADMM (possibly with linearization or preconditioning techniques) to derive an implementable
decentralized consensus algorithm 
\cite{Chang:2015a,Iutzeler:2016a,Jakovetic:2015a,Makhdoumi:2017a,Shi:2014a,Yuan:2016a}.
Most of these methods require each node to solve a local optimization problem 
every iteration before communication, and reach a convergence rate of $O(1/T)$
in terms of outer iteration (communication) number $T$ for general convex objective functions $\{f_i\}$. 
First-order methods based on decentralized gradient
descent require less computational cost
at individual nodes such that between two communications 
they only perform one step of a gradient descent-type update at the weighted average of previous iterates obtained from neighbors.
In particular, Nesterov's optimal gradient scheme 
is employed in decentralized gradient descent with diminishing step sizes to achieve 
rate of $O(1/T)$ in \cite{Jakovetic:2014a}, where an alternative gradient method
that requires excessive communications in each inner iteration is also developed
and can reach a theoretical convergence rate of $O(\log T/T^2)$, despite that
it seems to work less efficiently in terms of communications than the former in practice. 
A correction technique is developed for decentralized gradient 
descent with convergence rate as $O(1/T)$ with constant step size in \cite{Shi:2015a}, 
which results in a saddle-point algorithm as pointed out in \cite{Mokhtari:2015b}.
In \cite{Zhao:2015a}, the authors combine Nesterov's gradient scheme
and a multiplier-type auxiliary variable to obtain a fast optimality convergence rate of $O(1/T^2)$.
Other first-order decentralized methods have also been developed recently, such dual averaging \cite{Duchi:2012b}.
Additional constraints for primal variables in decentralized consensus optimization \eqref{eqn:deccon}
are considered in \cite{Yuan:2015a}.

In real-world decentralized computing, it is often difficult and sometimes impossible to coordinate all the nodes
in the network such that their computation and communication are perfectly synchronized.
One practical approach for such \textit{asynchronous consensus}
is using a broadcast scenario where in each (outer) iteration, one node in the network 
is assumed to wake up at random and broadcasts its value to neighbors
(but does not hear them back). A number of algorithms for broadcast consensus 
are developed, for instance, in \cite{Aysal:2009a,Iutzeler:2012a,Nedic:2015a,Nedic:2016a}.
{In particular, \cite{Nedic:2016a} develops a consensus optimization algorithm for \eqref{eqn:deccon} in
the setting where every iteration one node in the network broadcasts its value to the neighbors,
but there are no delays in (sub)gradients during their updates.}
Another important issue in the asynchronous setting is that nodes may have to use out-of-date
(stale) gradient
information during updates \cite{Nedic:2009b,Wu:2016a}.
This delayed scenario in gradient descent is considered in a distributed but not decentralized setting
in \cite{Agarwal:2011a,Li:2013b,Sra:2016a,Zhang:2016a}.
In addition, analysis of stochastic gradient in distributed computing is also 
carried out in \cite{Agarwal:2011a,Shamir:2014a}.
In \cite{Feyzmahdavian:2014a}, linear convergence rate of optimality is derived for strongly convex
objective functions with delays. Extending \cite{Agarwal:2011a}, a \textit{fixed} delay at all nodes 
is considered in dual averaging \cite{Li:2015a} and gradient descent \cite{Wang:2015a}
in a decentralized setting, but they
did not consider more practical and useful \textit{random} delays,
and there are no convergence rates on node consensus provided in these papers. 
In \cite{Wu:2016a}, both random delays in communications and gradients are considered,
however, no convergence rate is established in such setting.

\subsection{Contributions}
The contribution of this paper is in three phases. 

First, we consider a general decentralized consensus
algorithm with randomly delayed and stochastic gradient (Section \ref{sec:algorithm}). 
In this case, the nodes do not need to be synchronized
and they may only have access to stale gradient information. This renders stochastic gradients
with random delays at different nodes in their gradient updates,
which is suitable for many real-world decentralized computing applications.

Second, we provide a comprehensive convergence analysis of the proposed algorithm (Section \ref{sec:convergence}). 
More precisely, we derive convergence rates for both the objective function (optimality) 
and disagreement (feasibility constraint of consensus), and show
their dependency on the characteristics of the problem, such as Lipschitz constants of 
(stochastic) gradients and spectral gaps
of the underlying network. 

Third, we conduct a number of numerical experiments on synthetic and real datasets
to validate the performance of the proposed algorithm (Section \ref{sec:experiment}). In particular, we examine the
convergence on synthetic decentralized least squares, robust least squares, and logistic regression problems.
We also present the numerical results on the reconstruction of several seismic images in
decentralized wireless sensor networks.

\subsection{Notations and assumptions}\label{subsec:notation}
In this paper, all vectors are column vectors unless otherwise noted. We denote by $x_i(t)\in\mathbb{R}^n$ the
estimate of node $i$ at iteration $t$, and $x(t)=(x_1(t),\dots,x_m(t))^{\top}\in\mathbb{R}^{m\times n}$.
We denote $\|x\|\equiv\|x\|_2$ if $x$ is a vector and $\|x\|\equiv\|x\|_F$ if $x$ is a matrix,
which should be clear by the context.
For any two vectors of same dimension, $\langle x,y\rangle$ denotes their inner product,
and $\langle x,y\rangle_{Q}:=\langle x,Qy\rangle$ for symmetric positive semidefinite matrix $Q$.
For notation simplicity, we use $\langle x,y\rangle=\sum_{i=1}^{m}\langle x_i,y_i\rangle$
where $x_i$ and $y_i$ are the $i$-th row of the $m\times n$ matrices $x$ and $y$ respectively.
Such matrix inner product is also generalized to $\langle x,y\rangle_Q$ for matrices $x$ and $y$.
In this paper, we set the domain $X:=\{x\in\mathbb{R}^n:\|x\|_\infty\leq R\}$ for some $R>0$,
which can be thought of as the maximum pixel intensity in reconstructed images for instance.
We further denote $\calX := X^m\subset \mathbb{R}^{m\times n}$.

For each node $i$, we define $f_i(x):=\Ex_{\xi_i}[F_i(x;\xi_i)]$ as the expectation of objective function,
and $g_i(t):=\nabla F_i(x(t);\xi_i(t))$ (here the gradient $\nabla$ is taken with respect to $x$) 
is the stochastic gradient at $x_i(t)$ at node $i$.
We let $\tau_i(t)$ be the delay of gradient at node $i$ in iteration $t$, and $\tau(t)=(\tau_1(t),\dots,\tau_m(t))^{\top}$.
We write $f(x(t))$ in short for $\sum_{i=1}^{m}f_i(x_i(t))\in\mathbb{R}$,
$x(t-\tau(t))$ for $(x_1(t-\tau_1(t)),\dots,x_m(t-\tau_m(t)))^{\top}\in\mathbb{R}^{m\times n}$,
and
$g(t-\tau(t))$ for $(g_1(t-\tau_1(t)),\dots,g_m(t-\tau_m(t)))^{\top}\in\mathbb{R}^{m\times n}$.
We assume $f_i$ is continuously differentiable, $\nabla f_i$ has Lipschitz constant $L_i$,
and denote $L:=\max_{1\leq i\leq m}L_i$. 

Let $x^*\in\mathbb{R}^n$ be a solution of \eqref{eqn:deccon}, we
denote $\bfone(x^*)^{\top}$ simply by $x^*$ in this paper which is clear by the context,
for instance $f(x^*)=f(\bfone (x^*)^{\top})=\sum_{i=1}^{m}f_i(x^*)$.
Furthermore, we let $y(T):=(1/T) \sum_{t=1}^{T}x(t+1)$ be the running average of $\{x(t+1):1\leq t\leq T\}$,
and $z(T):=(1/m)\sum_{i=1}^m y(T)$ be the consensus average of $y(T)$.
We denote $J=(1/m)\bfone\bfone^{\top}$, then $z(T)=Jy(T)$.
Note that for all $T$, $z(T)$ is always consensual but $x(T), y(T)$ may not be.


An important ingredient in decentralized gradient descent is the mixing matrix $W=[w_{ij}]$
in \eqref{eqn:algupd_node}. For the algorithm to be implementable in practice,
$w_{ij}>0$ if and only if $(i,j)\in \mathcal{E}$. In this paper, we assume that $W$ is
symmetric and $\sum_{j=1}^{m} w_{ij}=1$ for all $i$, hence $W$ is doubly
stochastic, namely $W\bfone=\bfone$ and $\bfone^{\top}W=\bfone^{\top}$ where $\bfone=(1,\dots,1)^{\top}\in\mathbb{R}^m$. 
With the assumption that the network $\mathcal{G}$ is simple and connected,
we know $\|W\|_2=1$ and eigenvalue $1$ of $W$ has multiplicity $1$
by the Perron-Frobenius theorem \cite{Lovasz:1993a}.
As a consequence, $Wx=x$ if and only if $x$ is consensual, i.e., $x=c\bfone$ for some $c\in\mathbb{R}$.
We further assume $W\succeq 0$ (otherwise use $\frac{1}{2}(I+W)\succeq 0$ since stochastic matrix $W$ has spectral radius 1).
Given a network $\mathcal{G}$, there are different ways to design the mixing matrix $W$. For some optimal choices of $W$,
see, e.g., \cite{Sayed:2013b,Xiao:2004a}.

Now we make several assumptions that are necessary in our convergence analysis.
\begin{enumerate}
\item The network $\mathcal{G}=(\mathcal{V},\mathcal{E})$ is undirected, simple, and connected.
\item For all $i$ and $x$, the stochastic gradient is unbiased, i.e., $\Ex_{\xi_i}[\nabla F_i(x;\xi_i)]=\nabla f_i(x)$,
and $\Ex_{\xi_i}[\|\nabla F_i(x;\xi_i)-\nabla f_i(x)\|^2]\leq \sigma^2$ for some $\sigma>0$.
\item The delays $\tau_i(t)$ may follow different distributions at different nodes, but their second moments
are assumed to be uniformly bounded, i.e., there exists $B>0$ such that 
$\Ex[\tau_i(t)^2]\leq B^2$ for all $i=1,\dots,m$ and iteration $t$.
\end{enumerate}
Since the domain $X$ is compact and $\nabla f_i$ are all Lipschitz continuous,
we know $\|\nabla f_i\|$ is uniformly bounded.
Furthermore, $\Ex[\|\nabla F_i(\cdot,\xi_i)\|] \leq \Ex[\|\nabla F_i(\cdot,\xi_i)-\nabla f_i(\cdot)\|] + \|\nabla f_i(\cdot)\|
\leq \sigma + \|\nabla f_i(\cdot)\|$, we know $\Ex[\|\nabla F_i(\cdot,\xi_i)\|]$ is also uniformly bounded.
Therefore, we denote by $G>0$ the uniform bound such that
$\|\nabla f_i\|, \Ex[\|\nabla F_i(\cdot,\xi_i)\|] \leq G$ for all $i$.
We also assume that the random delay $\tau_i(t)$ and error of inexact gradient
$\epsilon_i(t):=g_i(t)-\nabla f_i(x(t))$ are independent.

%
%

\section{Algorithm}\label{sec:algorithm}
Taking the delayed stochastic gradient and the constraint that nodes can only communicate
with immediate neighbors, we propose the following decentralized delayed stochastic 
gradient descent method for solving \eqref{eqn:deccon}. Starting from an initial 
guess $\{x_i(0):i=1,\dots,m\}$, each node $i$ performs the following updates iteratively:
\begin{equation}\label{eqn:algupd_node}
x_i(t+1)=\Pi_X \sbr[3]{\sum_{j=1}^m w_{ij}x_j(t)-\alpha(t)g_i(t-\tau_i(t))}.
\end{equation}
Namely, in each iteration $t$, the nodes exchange their most recent $x_i(t)$ with
their neighbors. Then each node takes weighted average of the received local copies 
using weights $w_{ij}$ and performs a gradient descent type update using a stochastic
gradient $g_i(t-\tau_i(t))$ with delay $\tau_i(t)$ and step size $\alpha(t)$, and projects the result onto $X$.
In addition, each node $i$ tracks its own running average $y_i(t)=(1/t)\cdot\sum_{s=1}^t x_i(s+1)$
by simply updating $y_i(t)=(1-1/t)\cdot y_i(t-1)+(1/t)\cdot x_i(t+1)$ in iteration $t$.


Following the matrix notation in Section \ref{subsec:notation}, the iteration \eqref{eqn:algupd_node} can be written as
\begin{equation}\label{eqn:algupd}
x(t+1)=\proj_{\calX}[Wx(t)-\alpha(t) g(t-\tau(t))].
\end{equation}
%
Here the projection $\proj_\calX$ is accomplished by each node projecting to $X$
due to the definition of $X$ in Section \ref{subsec:notation},
which does not require any coordination between nodes.
Note that the update \eqref{eqn:algupd} is also equivalent to
\begin{equation}\label{eqn:alg}
x(t+1)= \argmin_{x \in \calX} \left\{ \langle g(t-\tau(t)),x\rangle + \frac{1}{2\alpha(t)}\|x-Wx(t)\|^2\right\}.
\end{equation}
In this paper, we may refer to the proposed decentralized delayed stochastic gradient descent 
algorithm by any of \eqref{eqn:algupd_node}, \eqref{eqn:algupd}, and \eqref{eqn:alg}
since they are equivalent.

\section{Convergence Analysis}\label{sec:convergence}
In this section, we provide a comprehensive convergence analysis of the
proposed algorithm \eqref{eqn:alg} by employing a proper step size policy. 
In particular, we derive convergence rates for both of
the disagreement (Theorem \ref{thm:cssrate}) and objective function value (Theorem \ref{thm:z_opt}).

\begin{lemma}\label{lemma:proj}
For any $x\in\mathbb{R}^{m\times n}$, its projection onto $\calX$ yields nonincreasing disagreement. That is
\begin{equation}
\|(I-J)\proj_{\calX}(x)\|\leq \|(I-J)x\|.
\end{equation}
\end{lemma}
\begin{proof}
See Appendix \ref{app:proof_proj}.
\end{proof}

\begin{lemma}\label{lemma:seqrate}
Let $c_1\geq0$ and $c_2>0$, and define $\alpha(t)=1/(c_1+c_2\sqrt{t})$.
Then for any $\lambda\in (0,1)$ there is
\begin{equation}
\sum_{s=0}^{t-1}\alpha(s)\lambda^{t-s-1}\leq \frac{\sqrt{\pi}\lambda^{-2}}{c_2\sqrt{t} \log(\lambda^{-1})}= O\left(\frac{1}{\sqrt{t}}\right)
\end{equation}
for all $t=1,2,\dots.$
\end{lemma}
\begin{proof}
See Appendix \ref{app:proof_seqrate}.
\end{proof}

Now we are ready to prove the convergence rate of disagreement 
in $x(t)$ and $y(t)$. In particular, we show that $(\sum_{i=1}^m \|x_i(t)-\bar{x}(t)\|^2)^{1/2}$
decays at the rate of $O(1/\sqrt{t})$, where $\bar{x}(t)=(1/m)\sum_{i=1}^m x_i(t)$.
The same convergence rate holds for the disagreement of running average $y(t)$.
More specifically, these convergence rates are given by the bounds 
in the following theorem.
\begin{theorem}\label{thm:cssrate}
Let $\{x(t)\}$ be the iterates generated by Algorithm 
\eqref{eqn:alg} with $\alpha(t)=[2(L+\eta\sqrt{t})]^{-1}$
for some $\eta>0$, 
and $\lambda=\|W-J\|$. Then $\lambda$ is the second largest eigenvalue of $W$
and hence $\lambda\in(0,1)$. Moreover, the disagreement of $x(t)$ is bounded by
\begin{equation}\label{eqn:cssrate}
\Ex [\|(I-J)x(t)\|]\leq \sqrt{m}G\sum_{s=0}^{t-1}\alpha(s)\lambda^{t-s-1} 
\leq \frac{\sqrt{\pi m}G\lambda^{-2}}{\eta\sqrt{t}\log(\lambda^{-1})}=O\left(\frac{1}{\sqrt{t}}\right),
\end{equation}
and the disagreement of running average $y(T)=(1/m)\sum_{t=1}^{T} x(t+1)$ is bounded by
\begin{equation}\label{eqn:Ycssrate}
\Ex[\|(I-J)y(T)\|] \leq \frac{2\sqrt{\pi m}G\lambda^{-2}}{\eta\sqrt{T}\log(\lambda^{-1})}=O\left(\frac{1}{\sqrt{T}}\right).
\end{equation}
\end{theorem}

\begin{proof}
We first prove the bound on disagreement between $\{x_i(t):1\leq i\leq m\}$, i.e., 
\eqref{eqn:cssrate}, by induction. It is trivial to show 
that this bound holds for $t=1$. Assuming \eqref{eqn:cssrate} holds for $t$, we have
\begin{align}
\Ex[\|(I-J)x(t+1)\|] = &\ \Ex[\|(I-J)\proj_{\calX}(Wx(t)-\alpha(t)g(t-\tau(t)))\| ]\nonumber \\
\leq &\ \Ex[\|(I-J)(Wx(t)-\alpha(t)g(t-\tau(t)))\|]\\
\leq &\ \Ex[\|(I-J)Wx(t)\|]+\alpha(t)\Ex[\|(I-J)g(t-\tau(t))\|] \nonumber \\
\leq &\ \Ex[\|(I-J)Wx(t)\|]+\alpha(t)\sqrt{m}G \nonumber
\end{align}
where we used Lemma \ref{lemma:proj} in the first inequality,
and $\|I-J\|\leq1$ and $\Ex[\|g_i(t-\tau_i(t))\|]\leq G$ in the last inequality. Noting that
$J^2=J$ and $JW=WJ=J$, we have
\begin{equation*}
(W-J)(I-J)=(I-J)W.
\end{equation*}
Therefore, we obtain
\begin{align}
\Ex[\|(I-J)x(t+1)\|] \leq & \Ex[\|(I-J)Wx(t)\|]+\alpha(t)\sqrt{m}G \nonumber \\
= & \Ex[\|(W-J)(I-J)x(t)\|]+\alpha(t)\sqrt{m}G \nonumber \\
\leq & \Ex[\|(W-J)\|\|(I-J)x(t)\|]+\alpha(t)\sqrt{m}G \\
\leq & \lambda \sqrt{m}G \sum_{s=0}^{t-1}\alpha(s)\lambda^{t-s-1}+\alpha(t)\sqrt{m}G \nonumber \\
= & \sqrt{m}G\sum_{s=0}^{t}\alpha(s)\lambda^{t-s} \nonumber 
\end{align}
where we used the induction assumption for $t$ in the last inequality.
Applying Lemma \ref{lemma:seqrate} to the bound yields the second inequality in \eqref{eqn:cssrate},
which shows that $\Ex[\|(I-J)x(t)\|]$ decays at rate $O(1/\sqrt{t})$.

By convexity of $\|\cdot\|$ and definition of $y(T)$, we obtain that
\begin{equation}
\Ex[\|(I-J)y(T)\|]\leq \frac{1}{T}\sum_{t=1}^{T} \Ex[\|(I-J)x(t+1)\|]
\leq \frac{2\sqrt{\pi m}G\lambda^{-2}}{\eta\sqrt{T}\log(\lambda^{-1})}
\end{equation}
by applying \eqref{eqn:cssrate} and using $\sum_{t=1}^{T}\frac{1}{\sqrt{t}}\leq 2\sqrt{T}$.
Therefore the disagreement $\Ex[\|(I-J)y(T)\|]$ also decays at rate of $O(1/\sqrt{T})$.
\end{proof}

The convergence rate of disagreement also yields an estimate 
of differences between consecutive iterates $x(t)$ and $x(t+1)$, which is given by the following corollary.

\begin{corollary}\label{cor:adjrate}
Let $\{x(t)\}$ be the iterates generated by Algorithm \eqref{eqn:alg} with the settings of $\alpha(t)$,
$\lambda$, and $\eta$ same as in Theorem \ref{thm:cssrate}. Then there is
\begin{equation}\label{eqn:adjrate}
\Ex[\|x(t+1)-x(t)\|] \leq \frac{C}{\sqrt{t}},
\end{equation}
where $C:=\frac{\sqrt{m}G}{\eta}\sbr[1]{\frac{\sqrt{\pi} \lambda^{-2}}{\log(\lambda^{-1})}+\frac{1}{2}}$
is a constant independent of $t$.
\end{corollary}

\begin{proof}
See Appendix \ref{app:proof_adjrate}.
\end{proof}

From the estimate of difference between consecutive iterates, we can also bound the expected difference
between $x(t)$ and $x(t-\tau(t))$ as follows.
\begin{corollary}\label{cor:delayrate}
Let $\{x(t)\}$ be the iterates generated by Algorithm \eqref{eqn:alg} with the settings of $\alpha(t)$,
$\lambda$, and $\eta$ same as in Theorem \ref{thm:cssrate}. Then there is
\begin{equation}
\Ex[\|x(t)-x(t-\tau(t))\|] \leq C\del{\frac{\sqrt{2m}B}{\sqrt{t}}+\frac{4mB^2}{t}}=O\left(\frac{1}{\sqrt{t}}\right).
\end{equation}
where $C$ is the constant defined in Corollary \ref{cor:adjrate}.
In particular, if $t\geq 8mB^2$, there is $\Ex[\|x(t)-x(t-\tau(t))\|]\leq \frac{2\sqrt{2m}CB}{\sqrt{t}}$.
\end{corollary}
\begin{proof}
See Appendix \ref{app:proof_delayrate}.
\end{proof}

Without loss of generality and for sake of notation simplicity, 
we assume iteration number $t>8mB^2$ and $\Ex[\|x(t)-x(t-\tau(t))\|]\leq \frac{2\sqrt{2m}CB}{\sqrt{t}}$
in the remaining derivations. The decay rate $O(1/\sqrt{t})$
of $\Ex[\|x(t)-x(t-\tau(t))\|]$ is useful to estimate the convergence
rate of objective function value later.

\begin{lemma}\label{lemma:sum_inprod}
Let $\{x(t)\}$ be the iterates generated by Algorithm \eqref{eqn:algupd}, then 
the following inequality holds for all $T\geq 1$:
\begin{align}
\qquad \sum_{t=1}^{T} \Ex\left\langle \nabla f(x(t))-\nabla f(x(t-\tau(t))), x(t+1)-x^*\right\rangle
\leq 8\sqrt{2nLT} mRCB
\end{align}
where $C$ is the constant defined in Corollary \ref{cor:adjrate}.
\end{lemma}

\begin{proof}
See Appendix \ref{app:proof_sum_inprod}.
\end{proof}

Now we are ready to prove the convergence rate of objective function value.
We first present the estimate of this rate for running averages $y(t)$ in the following theorem.

\begin{theorem}\label{thm:y_opt}
Let $\{x(t)\}$ be the iterates generated by Algorithm \eqref{eqn:algupd} with 
$\alpha(t)= [2(L+\eta\sqrt{t})]^{-1}$ 
for some $\eta>0$, then
\begin{align}\label{eqn:y_opt}
\Ex [f(y(T))]- f(x^*)\leq \frac{L\DX^2}{T}+
\frac{K}{\sqrt{T}}=O\left(\frac{1}{\sqrt{T}}\right)
\end{align}
where $y(T)=(1/T) \sum_{t=1}^{T}x(t+1)$ is the running average of $\{x(t)\}$,
$\DX=2\sqrt{mn}R$ is the diameter of $\mathcal{X}$,
and $K:=\eta\DX^2+4\sqrt{2mL}\DX CB+(4m\sigma^2/\eta)$.
\end{theorem}

\begin{proof}
See Appendix \ref{app:proof_y_opt}.
\end{proof}

We have shown that the running average $y(T)$ makes the objective function
decay as in \eqref{eqn:y_opt}.
However, since each node $i$ obtains its own $y_i(T)$ which may not be consensual
(and the left hand side of \eqref{eqn:y_opt} could be negative),
we need to look at their consensus average $z(T)=(1/m)\sum_{i=1}^my_i(T)$ 
and the convergence rate of its objective function value. This is given in the following theorem.

\begin{theorem}\label{thm:z_opt}
Let $x(t)$ be generated by Algorithm \eqref{eqn:algupd_node} with $\alpha(t)=[2(L+\eta\sqrt{t})]^{-1}$
for some $\eta>0$. 
Let $y(T)=(1/T)\sum_{t=1}^{T} x(t+1)$ be the running average of $x(t)$ and 
$z(T)=Jy(T)=(1/m)\sum_{i=1}^m y_i(T)$ be the consensus average of $y(T)$, then
\begin{align}\label{eqn:optbound}
0\leq\Ex [f( z(T) )] - f(x^*) \leq
\frac{L\DX^2+2\sqrt{m}LC^2}{T}+
\frac{K+2\sqrt{m}CG}{\sqrt{T}}=O\left(\frac{1}{\sqrt{T}}\right)
\end{align}
where $C$ is defined as in Corollary \ref{cor:adjrate}, and 
$\DX$ and $K$ are defined as in Theorem \ref{thm:y_opt}.
\end{theorem}

\begin{proof} We first bound the difference between the function values
at the running average $y(T)$ and the consensus average $z(T)=Jy(T)$:
\begin{align}\label{eqn:absdiff_fYfZ}
& \quad \, f(y(T))-f(z(T)) = \sum_{i=1}^m (f_i(y_i(T))-f_i(z(T))) \nonumber \\
& \leq \sum_{i=1}^m \langle \nabla f_i(z(T)),y_i(T)-z(T)\rangle + \frac{L_i}{2}\|y_i(T)-z(T)\|^2 \\
& \leq \sqrt{m}G \|(I-J)y(T)\| + \frac{L}{2}\|(I-J)y(T)\|^2 \leq \frac{2\sqrt{m}CG}{\sqrt{T}} + \frac{2C^2L}{T}, \nonumber
\end{align}
where we used convexity of $f_i$ and Lipschitz continuity of $\nabla f_i$ in the first inequality, 
$\|\nabla f_i\|\leq G$ and convexity of $\|\cdot\|^2$ 
in the second inequality, and Theorem \ref{thm:cssrate} to get the last inequality.
%
Therefore, combining \eqref{eqn:absdiff_fYfZ} and \eqref{eqn:y_opt} from Theorem \ref{thm:y_opt},
we obtain the bound in \eqref{eqn:optbound}.
Note that $z(T)$ is consensus, so $f(z(T))\geq f(x^*)$ since $x^*$ is
a consensus optimal solution of \eqref{eqn:deccon}.
This completes the proof.
\end{proof}

In summary, we have showed that the running average $y_i(T)$,
which can be easily updated by each node $i$, 
yields convergence in optimality and consensus feasibility.
More precisely, Theorem \ref{thm:cssrate} implies that $\|y_i(T)-z(T)\|$ converges to $0$ at rate $O(1/\sqrt{T})$
for all nodes $i$ where $z(T)=(1/m)\sum_{i=1}^m y_i(T)$ is their consensus average,
and Theorem \ref{thm:z_opt} implies that $f(z(T))$ converges to $f(x^*)$ at rate of $O(1/\sqrt{T})$.
It is known that $O(1/\sqrt{T})$ is the optimal rate
for stochastic gradient algorithms in centralized setting, 
and hence these two Theorems suggest an encouraging 
fact that such rate can be retained even if the problem becomes much more complicated, 
i.e., the gradients are stochastic and delayed, and the computation is carried out in decentralized setting.
{To retain convergence in this complex setting, we employed a diminishing step size policy
as commonly used in stochastic optimization. Such step size policy results in a convergence rate
of $O(1/\sqrt{T})$ even without delays and randomness in gradients.
Furthermore, due to errors and uncertainties in delayed and stochastic gradients, the iterates
may be directed further apart from solution during computations. As a consequence, the
constant in the estimated convergence rate appears to depend on the bound of set $X$ rather than 
the distance between initial guess and solution set 
as in the setting with non-delayed and non-stochastic gradients.}

\section{Numerical Experiments}\label{sec:experiment}

In this section, we test algorithm \eqref{eqn:algupd_node}  
on decentralized consensus optimization problem \eqref{eqn:deccon}
with delayed stochastic gradients
using a number of synthetic and real datasets.
The structure of network $\Gcal=(\Vcal,\Ecal)$ and objective function in \eqref{eqn:deccon}
are explained for each dataset, followed by performance evaluation
shown in plots of objective function $f(z(T))$ and disagreement 
$\sum_{i=1}^m \|y_i(T)-z(T)\|^2$ versus the iteration number $T$,
where $y_i(T)=(1/T)\sum_{t=1}^{T}x_i(t+1)$ is the running average of
$x_i(t)$ in algorithm \eqref{eqn:algupd_node} at each node $i$, and
$z(T)=(1/m)\sum_{i=1}^{m} y_i(T)$ is the consensus average 
at iteration $T$. 

\subsection{Test on synthetic data}

We first test on three different types of objective functions using
synthetic datasets. In particular, we apply algorithm \eqref{eqn:algupd_node}  
to decentralized least squares, decentralized robust least squares,
and decentralized logistic regression problems
with different delay and stochastic error combinations.
Then we compare the performance of the algorithm with and without delays
and stochastic errors in gradients.
The performance of the algorithm on different network size $m$
and time comparison with synchronous algorithm are also presented.

In the first set of tests on three different objective functions, we
simulate a network of regular $5\times 5$ 2-dimensional (2D) lattice of size $m=25$.
We set dimension of unknown $x$ to $n=10$ and generate an $\hat{x}\in\mathbb{R}^n$
using MATLAB built-in function \texttt{rand}, and set the $\ell_\infty$ radius
of $X$ to $R=1$. For each node $i$, we generate matrices $A_i\in\mathbb{R}^{p_i\times n}$ with $p_i=5$
using \texttt{randn}, and normalize each column into unit $\ell_2$ ball in $\mathbb{R}^{p_i}$ for
$i=1,\dots,m$.
Then we simulate
$b_i=A_i\hat{x}+\epsilon_i$ where $\epsilon_i$ is generated by
\texttt{randn} with mean $0$ and standard deviation $0.001$.
For decentralized least squares problem, 
we set the objective function to $f_i(x)=(1/2)\|A_ix-b_i\|^2$ at node $i$.
Therefore the Lipschitz constant of $\nabla f_i$ is $L_i=\|A_i^{\top}A_i\|_2$,
and we further set $L=\max_{1\leq i\leq m}\{L_i\}$.
The initial guess $x_i(0)$ is set to $0$ for all $i$.
For each iteration $t$, the delay $\tau_i(t)$ at each node $i$ is uniformly
drawn from integers $1$ to $B$ with $B=5$, $10$ and $20$. 
For given $t$, the stochastic gradient is simulated by 
setting $\nabla F_i(x_i(t);\xi_i(t))=A_i^{\top}(A_ix_i(t)-b_i)+\xi_i(t)$ 
where $\xi_i(t)$ is generated by \texttt{randn} with mean $0$
and standard deviation $\sigma$ set to $0.01$ and $0.05$.
We run our algorithm using step size $\alpha(t)=1/(2L+2\eta\sqrt{t})$ with 
$\eta=0.01$.
The objective function $f(z(T))-f^*$ and disagreement 
$\sum_{i=1}^m \|y_i(T)-z(T)\|^2$ versus the iteration number $T$ are plotted in
the top row of Figure \ref{fig:synthetic},
where the reference optimal objective $f^*=\min_{x\in X} \sum_{i=1}^mf_i(x)$ is computed
using centralized Nesterov's accelerated gradient method \cite{Nesterov:1983a,Tseng:2008a}.
In the two plots, we observe that both $f(z(T))-f^*$
and disagreement $\sum_{i=1}^m \|y_i(T)-z(T)\|^2$ decays to 0
as justified by our theoretical analysis in Section \ref{sec:convergence}.
In general, we observe that delays with larger bound $B$
and/or larger standard deviation $\sigma$ in stochastic gradient 
yield slower convergence, as expected.

We also tested on two different objective functions: robust least squares 
and logistic regression.
In robust least squares, we apply \eqref{eqn:algupd_node} to the decentralized
optimization problem \eqref{eqn:deccon} where the objective function is set to
\begin{equation}
f_i(x):=\sum_{j=1}^{p_i}h_i^{j}(x),\ \mbox{where }h_i^{j}(x)=
\begin{cases}
\frac{1}{2}|(a_i^{j})^{\top}x-b_i^{j}|^2 & \mbox{if } |(a_i^{j})^{\top}x-b_i^{j}|\leq\delta\\
\delta (|(a_i^{j})^{\top}x-b_i^{j}|-\frac{\delta}{2}) & \mbox{if }|(a_i^{j})^{\top}x-b_i^{j}|>\delta
\end{cases}
\end{equation}
where $(a_i^{j})^{\top}\in\mathbb{R}^n$ is the $j$-th row of matrix $A_i\in\mathbb{R}^{p_i\times n}$,
and $b_i^{j}\in\mathbb{R}$ is the $j$-th component of $b_i\in\mathbb{R}^{p_i}$ at each node $i$. 
In this test, we simulate network $\Gcal=(\Vcal,\Ecal)$ and set
$A_i$, $b_i$, $m$, $n$, $R$, $x_i(0)$ the same way as in the decentralized 
least squares test above, and set the parameter of the Huber norm in the robust least squares $\delta=0.05$.
The stochastic gradient is given by 
$\nabla F_i(x;\xi_i(t))=\sum_{j=1}^{p_i}\nabla h_i^j(x)+\xi_i(t)$ 
where $\xi_i(t)$ is generated as before with 
$\sigma$ set to $0.01$ and $0.05$. Lipschitz constants $L_i$ and $L$ are determined as 
in the previous test.
The settings of $\eta$ and $\tau_i(t)$ remain the same as well.
The objective function $f(z(T))-f^*$ and disagreement $\sum_{i=1}^m \|y_i(T)-z(T)\|^2$
are plotted in the middle row of Figure \ref{fig:synthetic}.
In these two plots, we observe similar convergence behavior as
in the test on the decentralized least squares problem above.
For the decentralized logistic regression, we generate $\hat{x}$, $\epsilon_i$ and $A_i$ the same way as before,
and set $b_i=\text{sign}(A_i\hat{x}+\epsilon_i)\in\{\pm1\}^{p_i}$ ($\text{sign}(0):=1$).
Now the objective function $f_i$ at node $i$ is set to
\begin{equation}
f_i(x)=\sum_{j=1}^{p_i}\left(\log[1+\exp((a_i^{j})^{\top}x)]-b_i^j(a_i^{j})^{\top}x\right),
\end{equation}
where $(a_i^{j})^{\top}\in\mathbb{R}^n$ is the $j$-th row of matrix $A_i\in\mathbb{R}^{p_i\times n}$,
and $b_i^{j}\in\mathbb{R}$ is the $j$-th component of $b_i\in\mathbb{R}^{p_i}$.
Then we perform \eqref{eqn:algupd_node}
to solve this problem in the network $\Gcal$ above. 
{Since $\nabla^2f_i(x)=\sum_{j}[\exp((a_i^{j})^{\top}x)/(1+\exp((a_i^{j})^{\top}x))^2]\cdot a_i^{j}(a_i^{j})^{\top}
\leq (1/4)\cdot\sum_j a_i^{j}(a_i^{j})^{\top}=(1/4)\cdot A_i^{\top}A_i$, there is
$\|\nabla f_i(x)-\nabla f_i(x')\|\leq (1/4)\cdot \|A_i^{\top}A_i\| \|x-x'\|$ for all $x,x'\in\mathbb{R}^n$.
Therefore we set $L_i=\|A_i^{\top}A_i\|_2/4$.}
The settings of the delay $\tau_i(t)$, $\eta$, and initial value 
$x_i(0)$ remain the same as before. The stochastic error level $\sigma$ is set to $0.1$ and $0.5$.
The objective function $f(z(T))-f^*$ and disagreement $\sum_{i=1}^m \|y_i(T)-z(T)\|^2$
are plotted in the bottom row of Figure \ref{fig:synthetic},
where similar convergence behavior as in the previous tests can be observed.

\begin{figure}[t!]
\centering
\includegraphics[width=.45\textwidth]{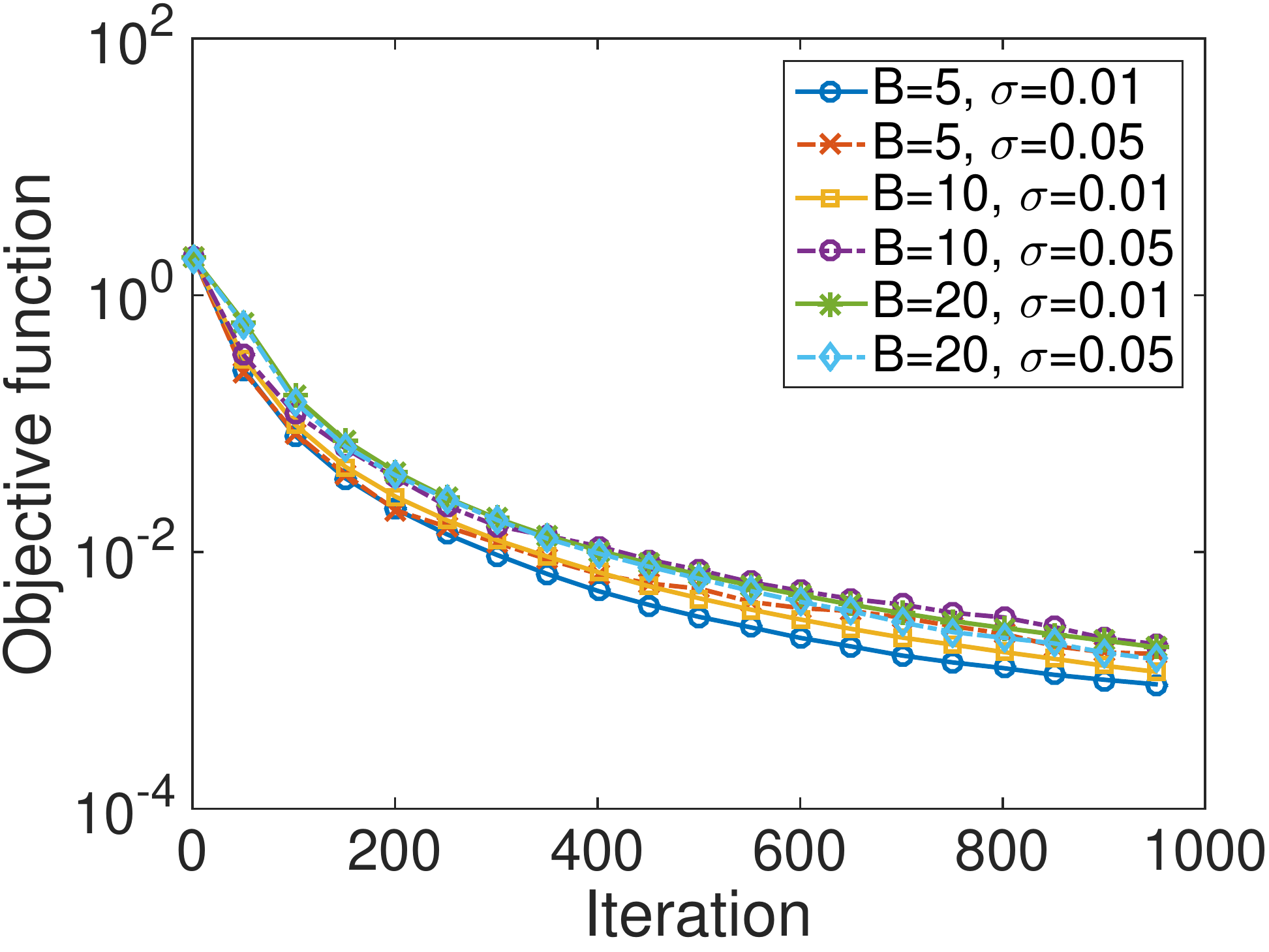}
\includegraphics[width=.45\textwidth]{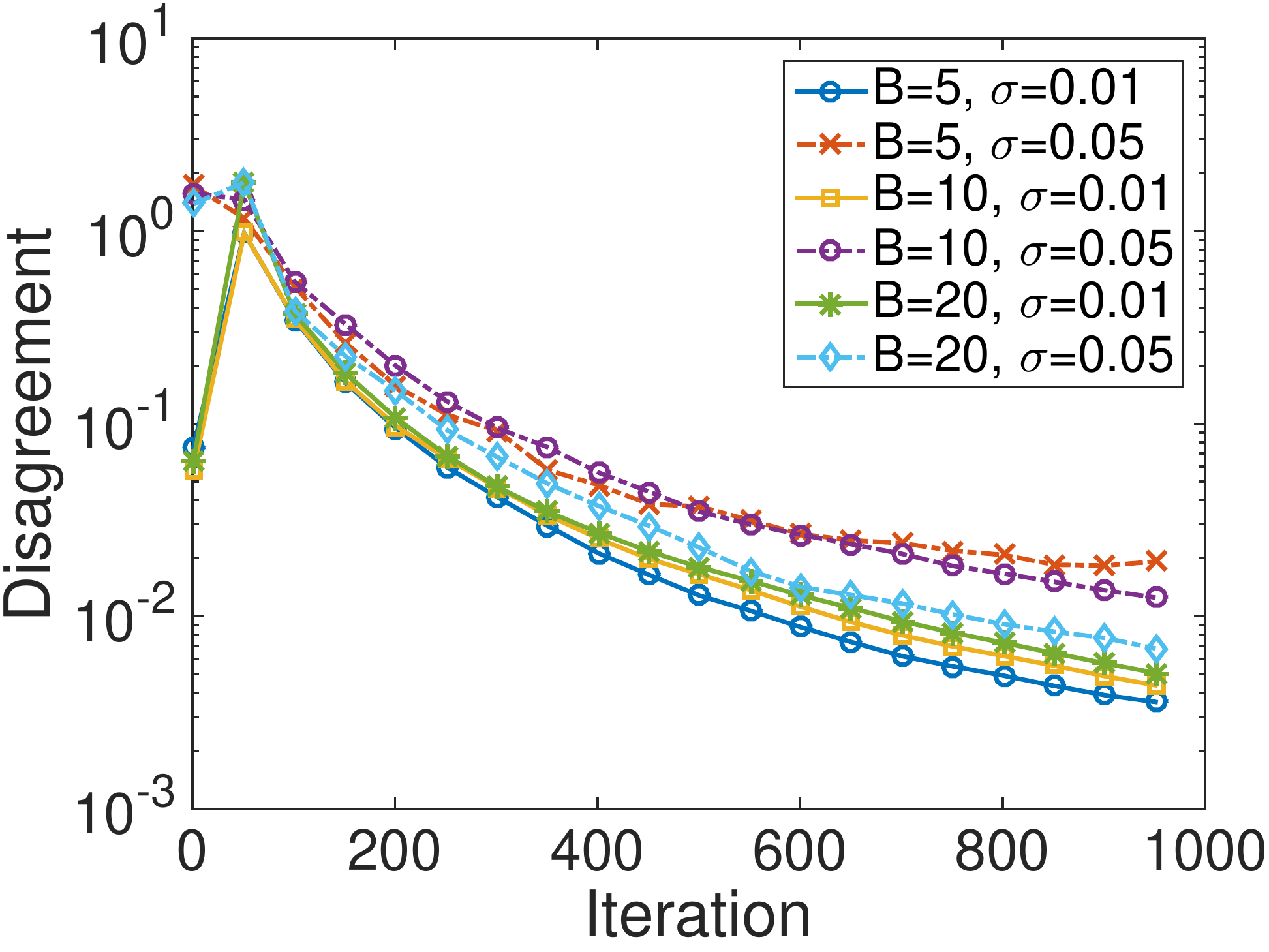}
\includegraphics[width=.45\textwidth]{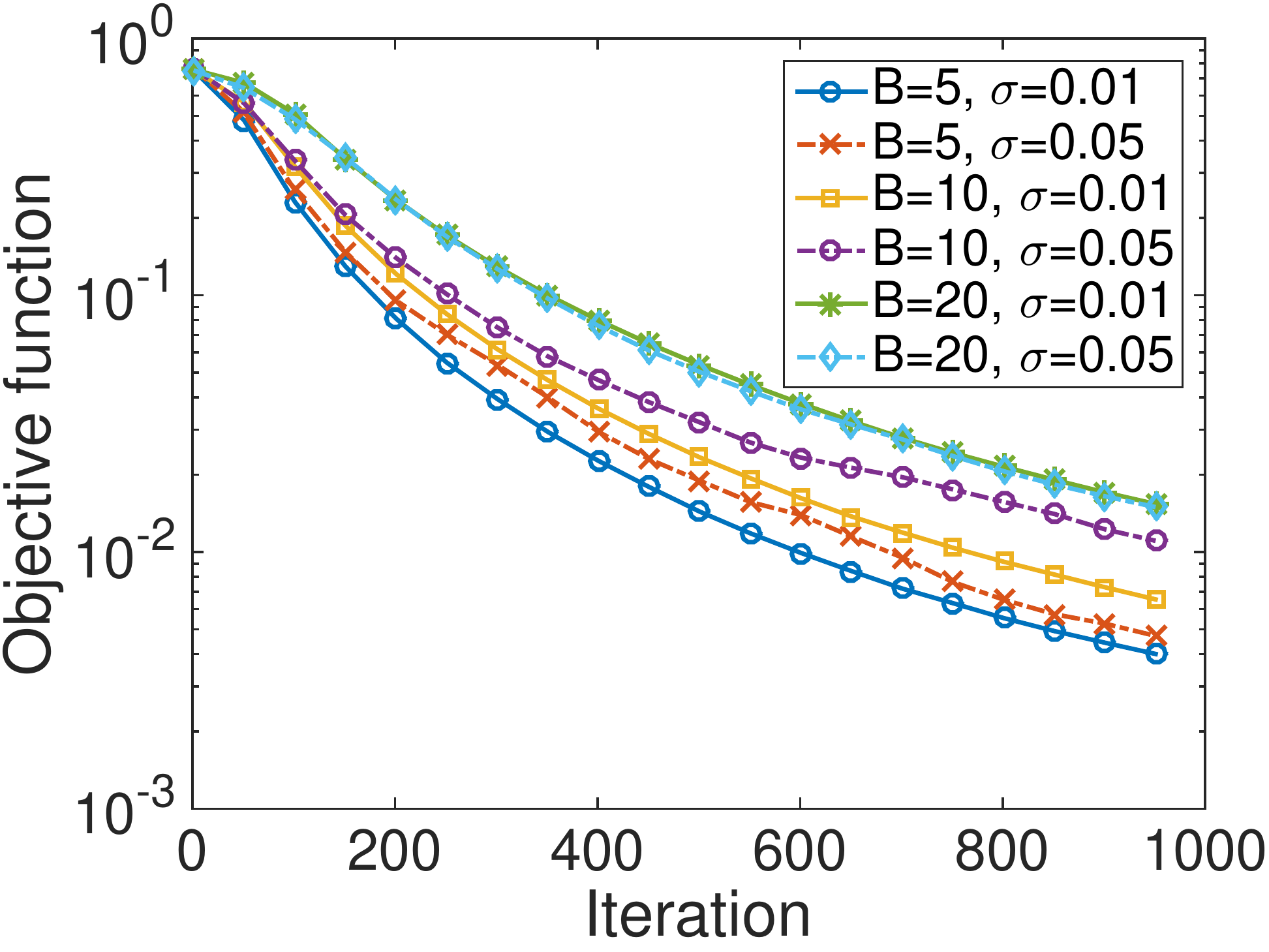}
\includegraphics[width=.45\textwidth]{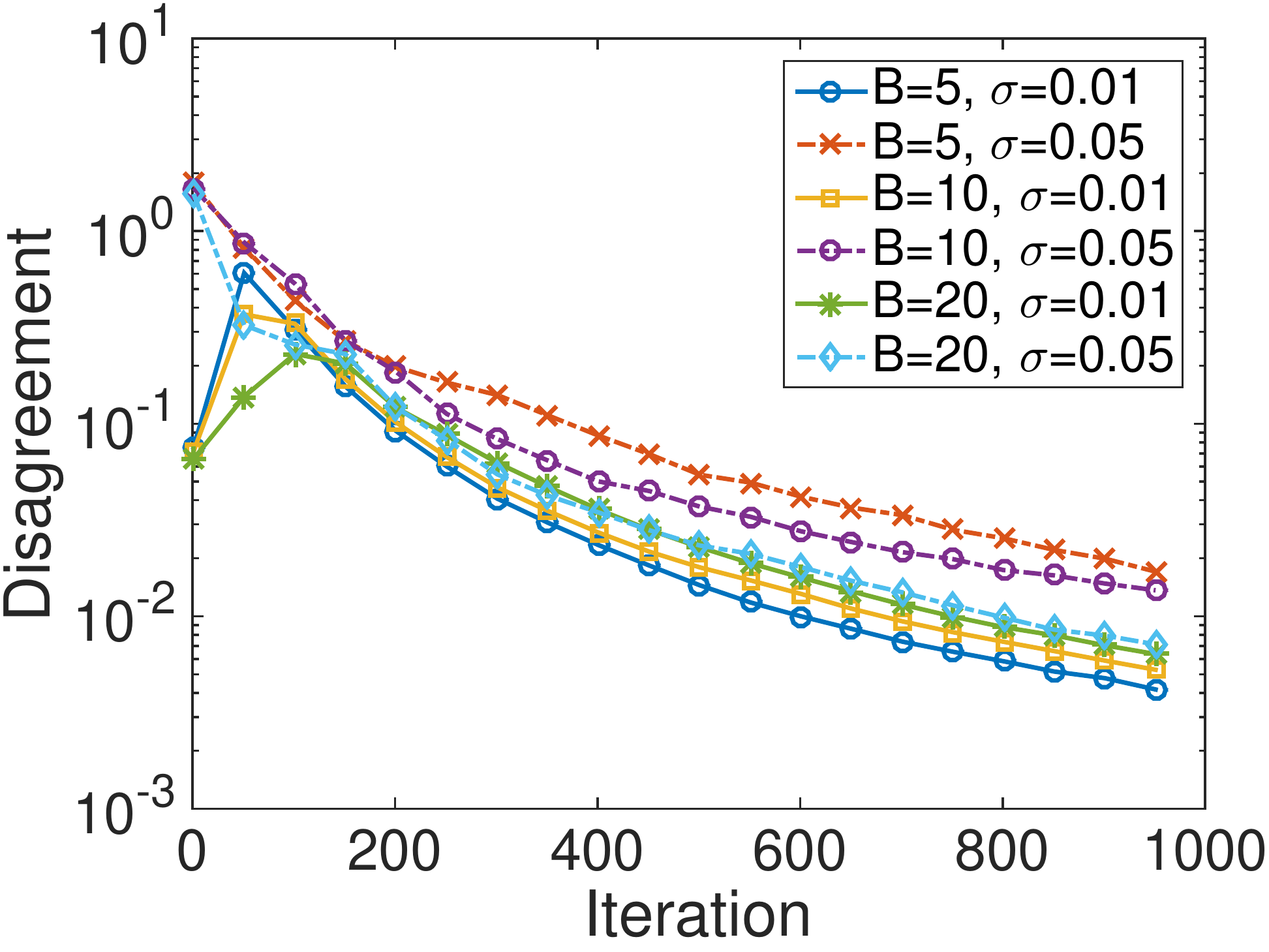}
\includegraphics[width=.45\textwidth]{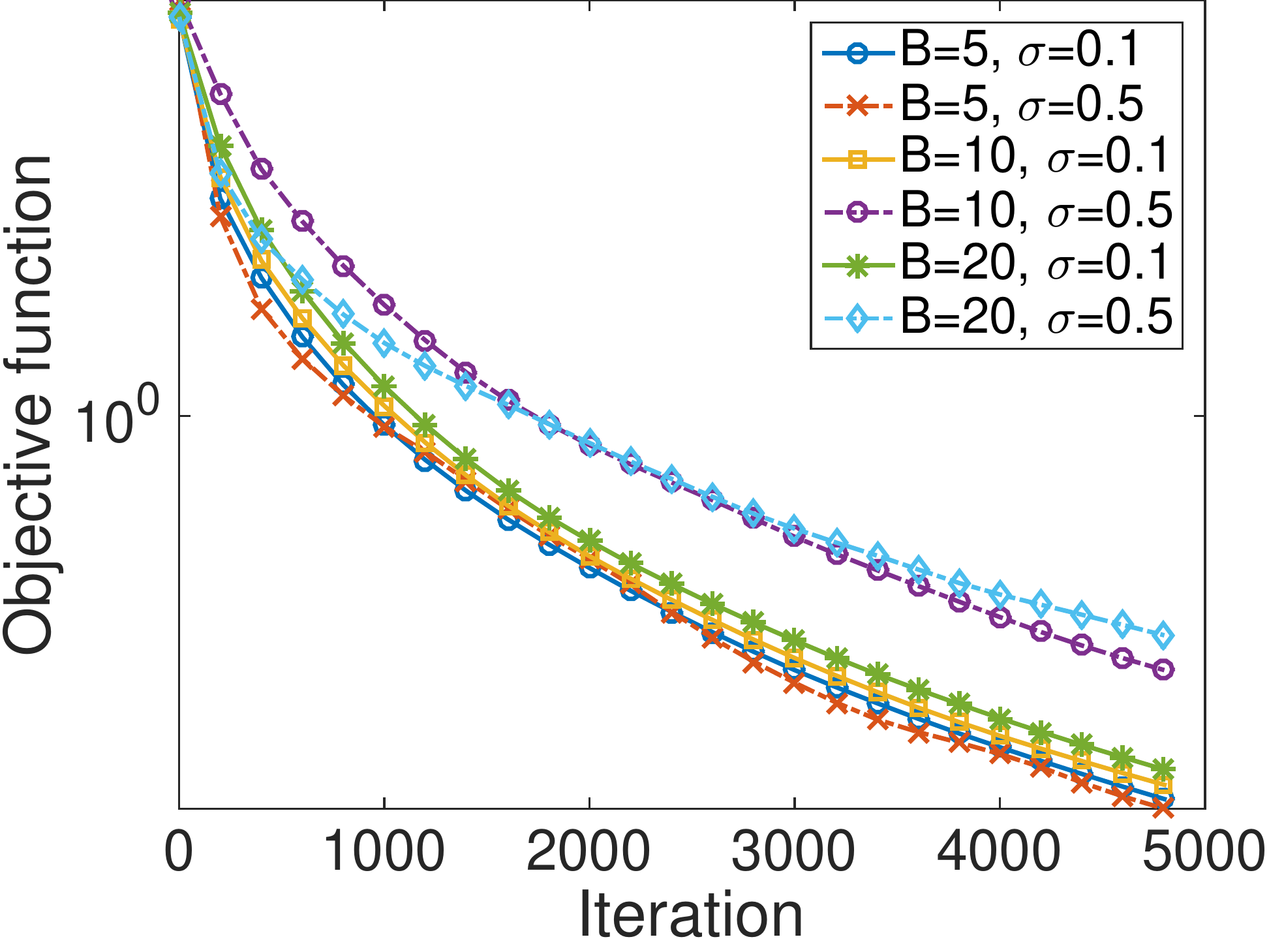}
\includegraphics[width=.45\textwidth]{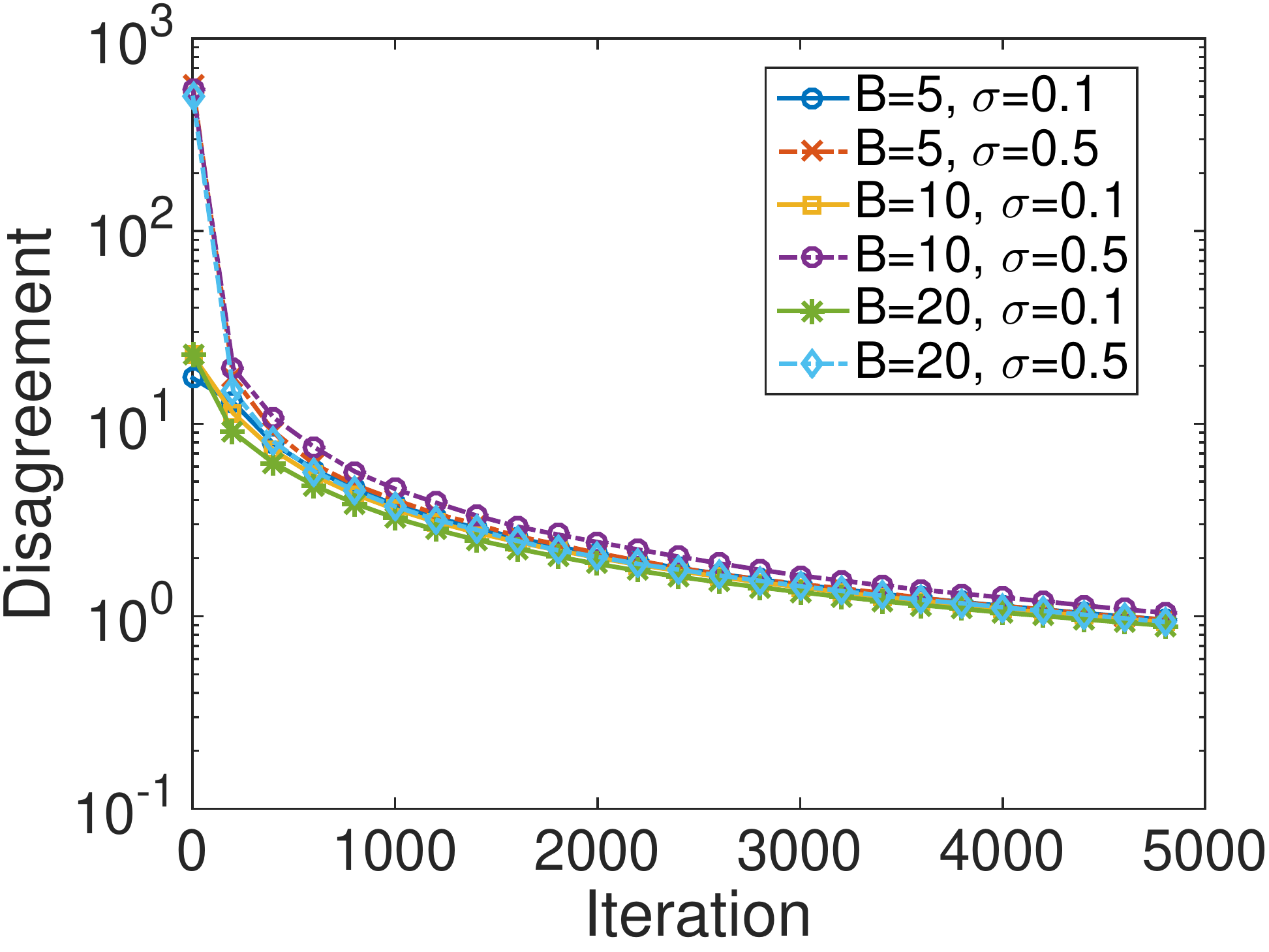}
\caption{Test on synthetic decentralized least-squares (top), robust least-squares (middle),
and logistic regression (bottom) for different
levels of delay $B$ and standard deviation in stochastic
gradient $\sigma$. Left: objective function $f(z(T))-f^*$ versus
iteration number $T$, where $f^*=f(x^*)$ is the optimal
value. Right: disagreement $\sum_{i=1}^m \|y_i(T)-z(T)\|^2$ versus
iteration number $T$.}
\label{fig:synthetic}
\end{figure}

We also compared the performance of decentralized gradient descent method
with and without delay and stochasticity in the gradients.
In this test, we synthesized networks and data in the same way as
in the decentralized least squares test above.
In addition, we plotted the result of $\tau_i(t)=0$ for all $i=1,\dots,m$ and 
$\sigma=0$ is for comparison.
These results are shown in the top row of Figure \ref{fig:compare},
The objective function value (top left) and disagreement (top right)
both decay sightly faster when there are no delay and stochastic 
error as shown in Figure \ref{fig:compare}, which is within expectations.
We further tested the performance when the network size varies. 
In this experiment, we used four 2D lattice networks, with sizes $m=5^2,10^2,15^2,20^2$.
The size of $x$ and $A_i$ at each node are the same as before.
The objective function value (middle left) and disagreement (middle right)
both decays, while it appears that network with smaller size decays faster, as shown in Figure \ref{fig:compare}.
{To demonstrate effectiveness of asynchronous consensus, we 
applied EXTRA \cite{Shi:2015a}, a state-of-the-arts synchronous decentralized consensus optimization method,
to the same data generated in decentralized least squares problem with network size $m=100$ and 
$\sigma=0$ (no stochastic error in gradients).
We draw computing times of these 100 nodes as
independent random variables between $[.001,.500]$ms every gradient evaluation. The synchronous 
algorithm EXTRA needs to wait for the slowest node to finish computation and then start a new iteration,
whereas in the asynchronous algorithm \eqref{eqn:algupd_node} the nodes communicate
with neighbors every 0.01ms using updates obtained by delayed gradients .
We plotted the objective function $f(z(T))-f^*$ and disagreement  $\sum_{i=1}^m \|y_i(T)-z(T)\|^2$
versus running time in the bottom row of Figure \ref{fig:compare},
which show that the asynchronous updates can be more time efficient by not waiting 
for slowest node in each iteration.}

\begin{figure}[t!]
\centering
\includegraphics[width=.45\textwidth]{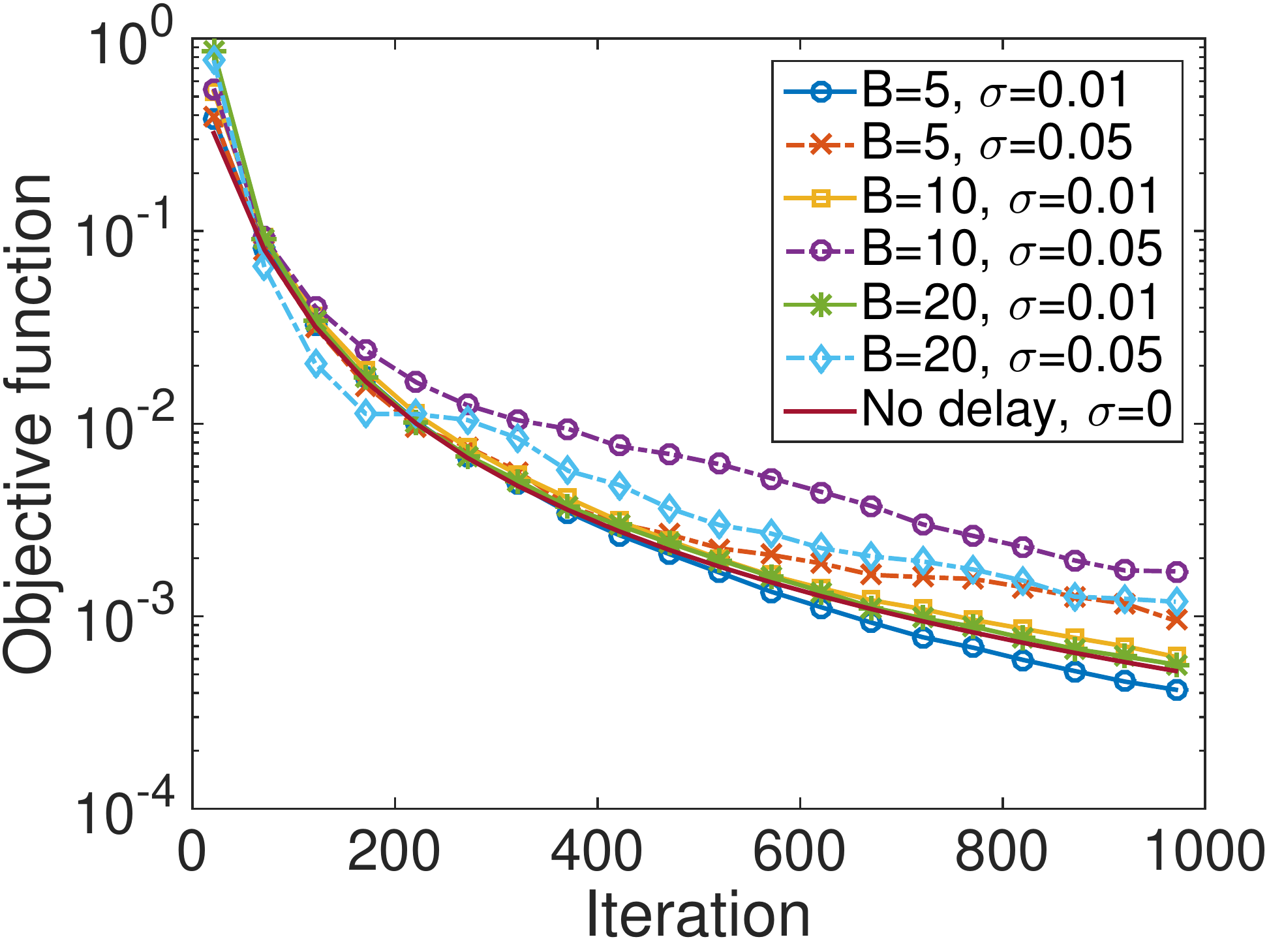}
\includegraphics[width=.45\textwidth]{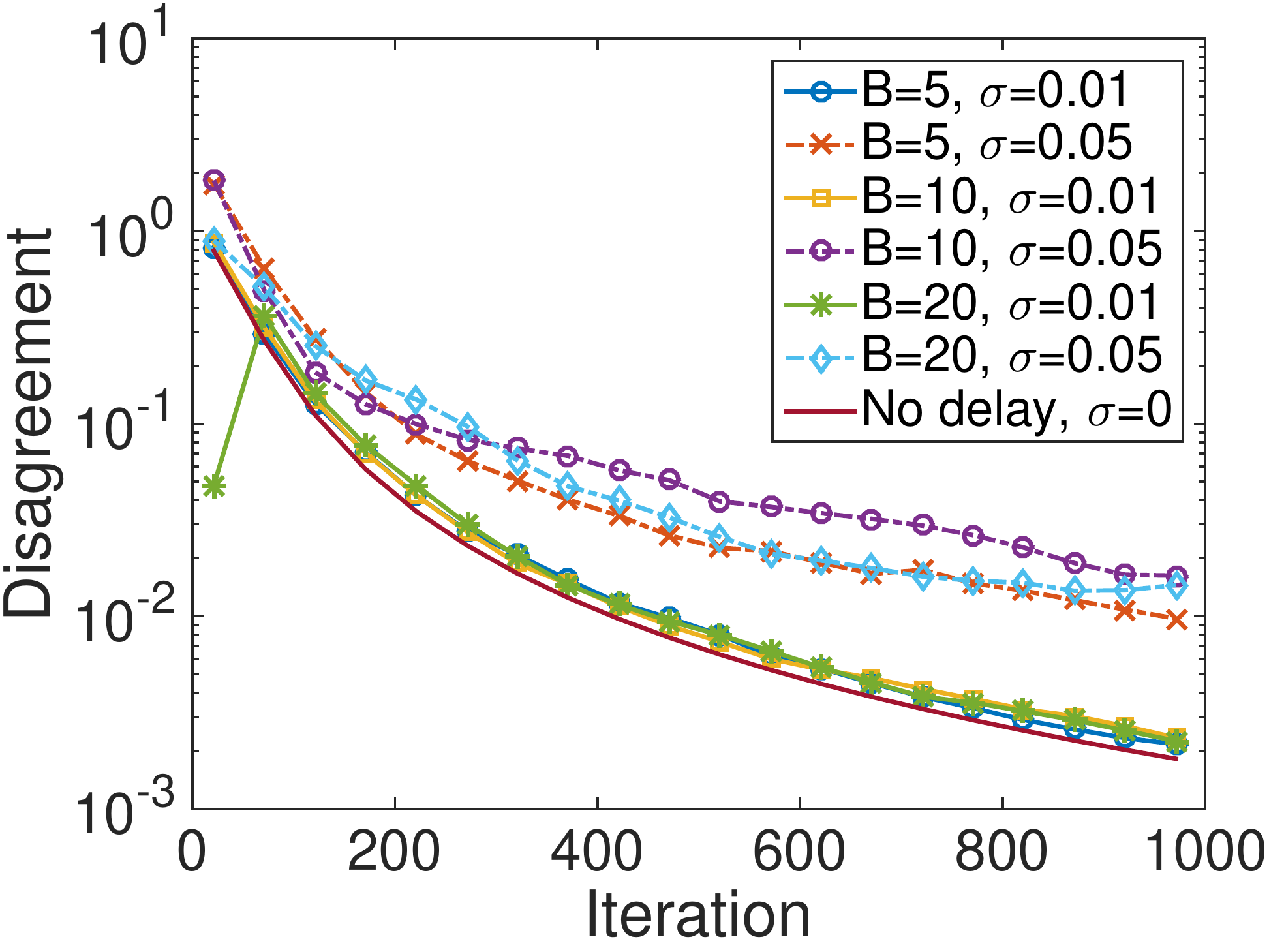}
\includegraphics[width=.45\textwidth]{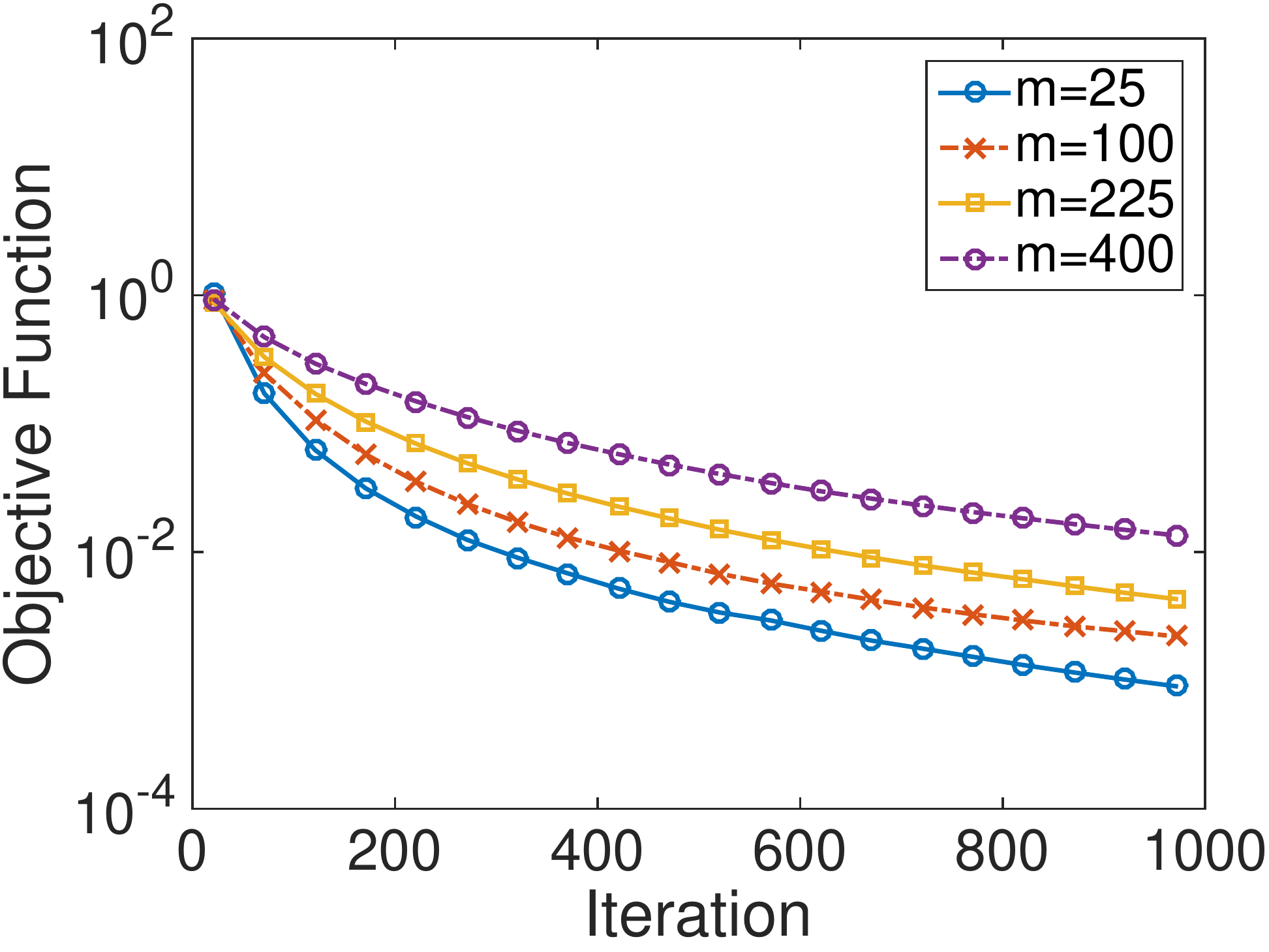}
\includegraphics[width=.45\textwidth]{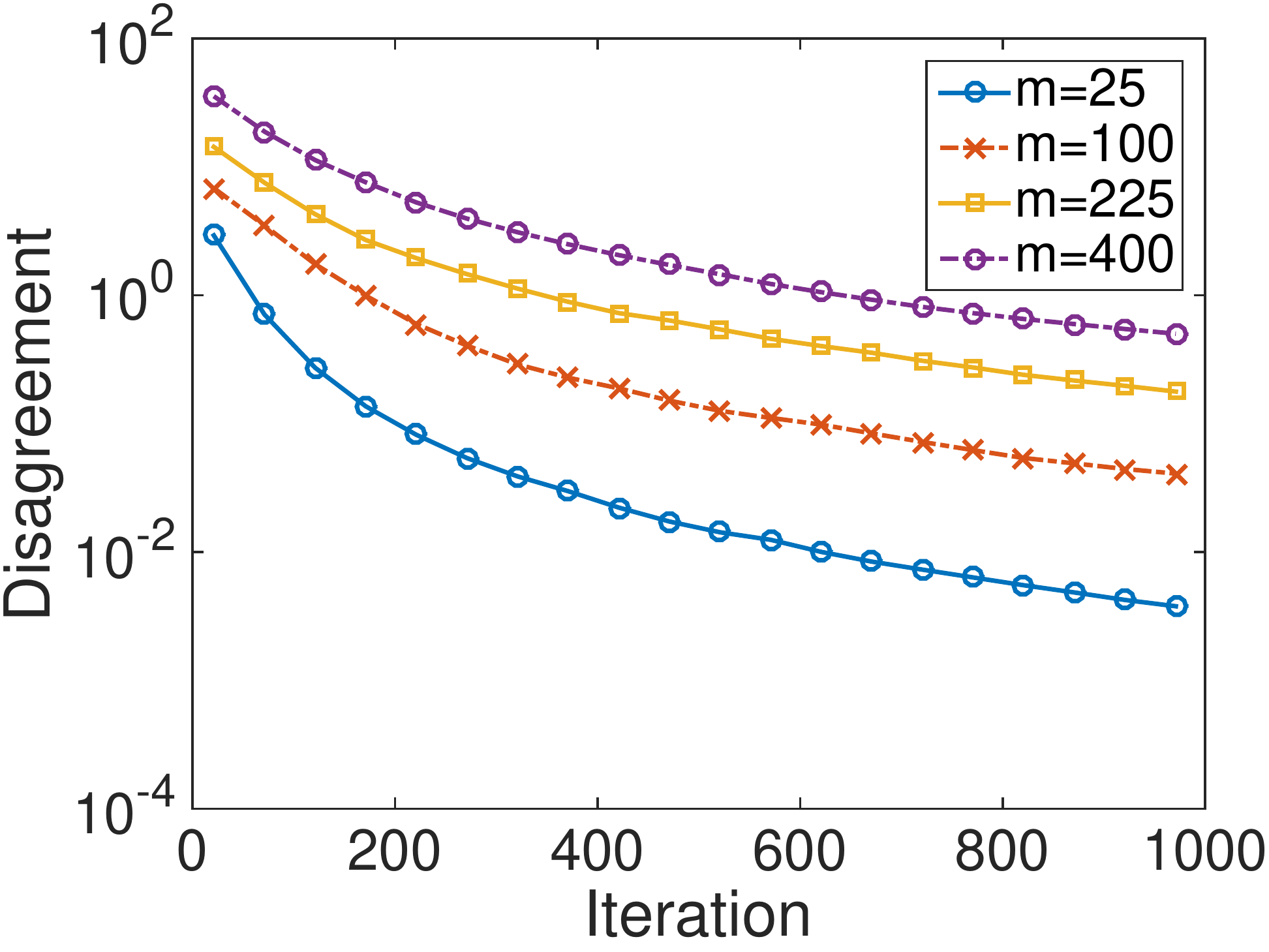}
\includegraphics[width=.45\textwidth]{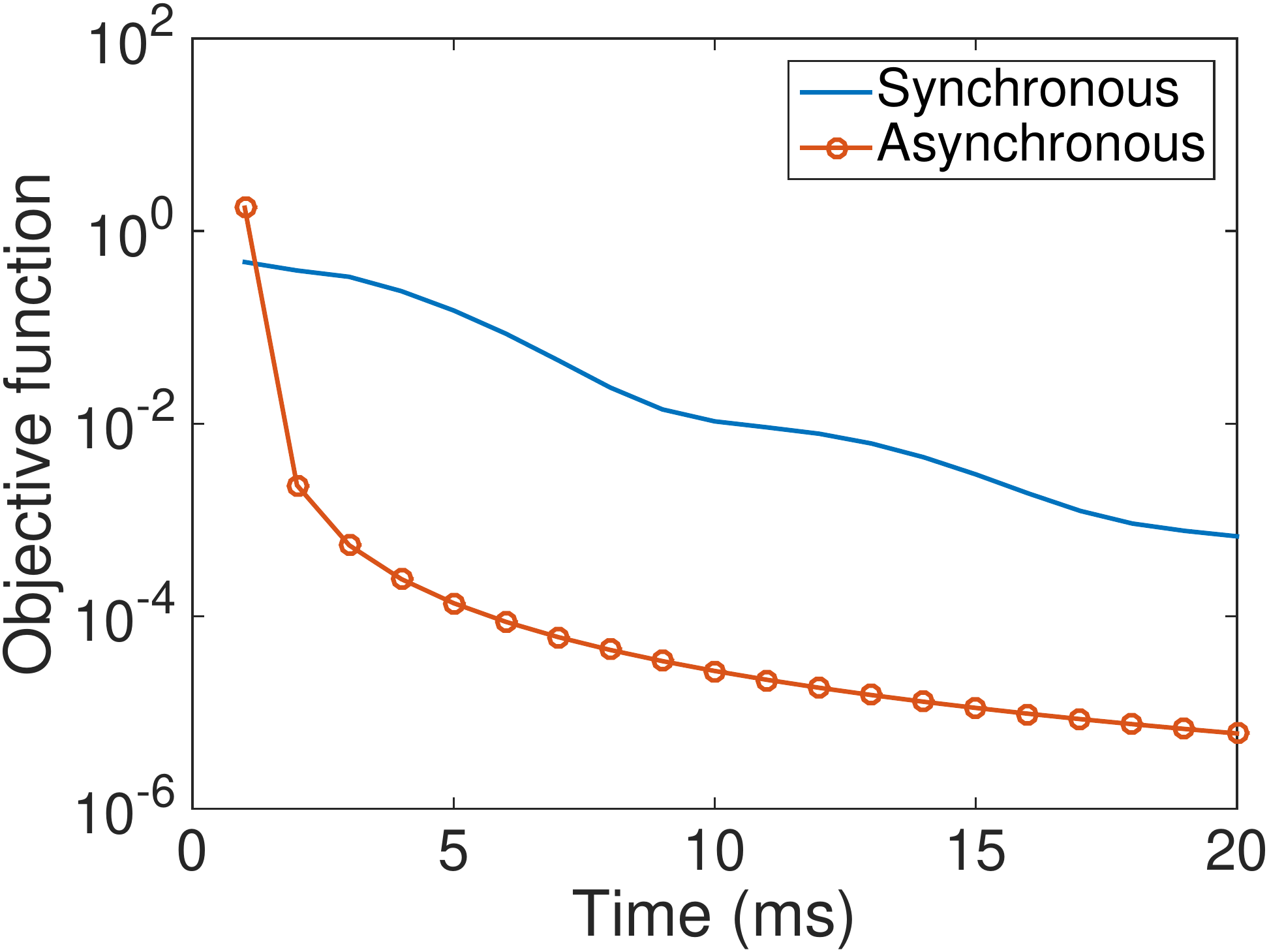}
\includegraphics[width=.45\textwidth]{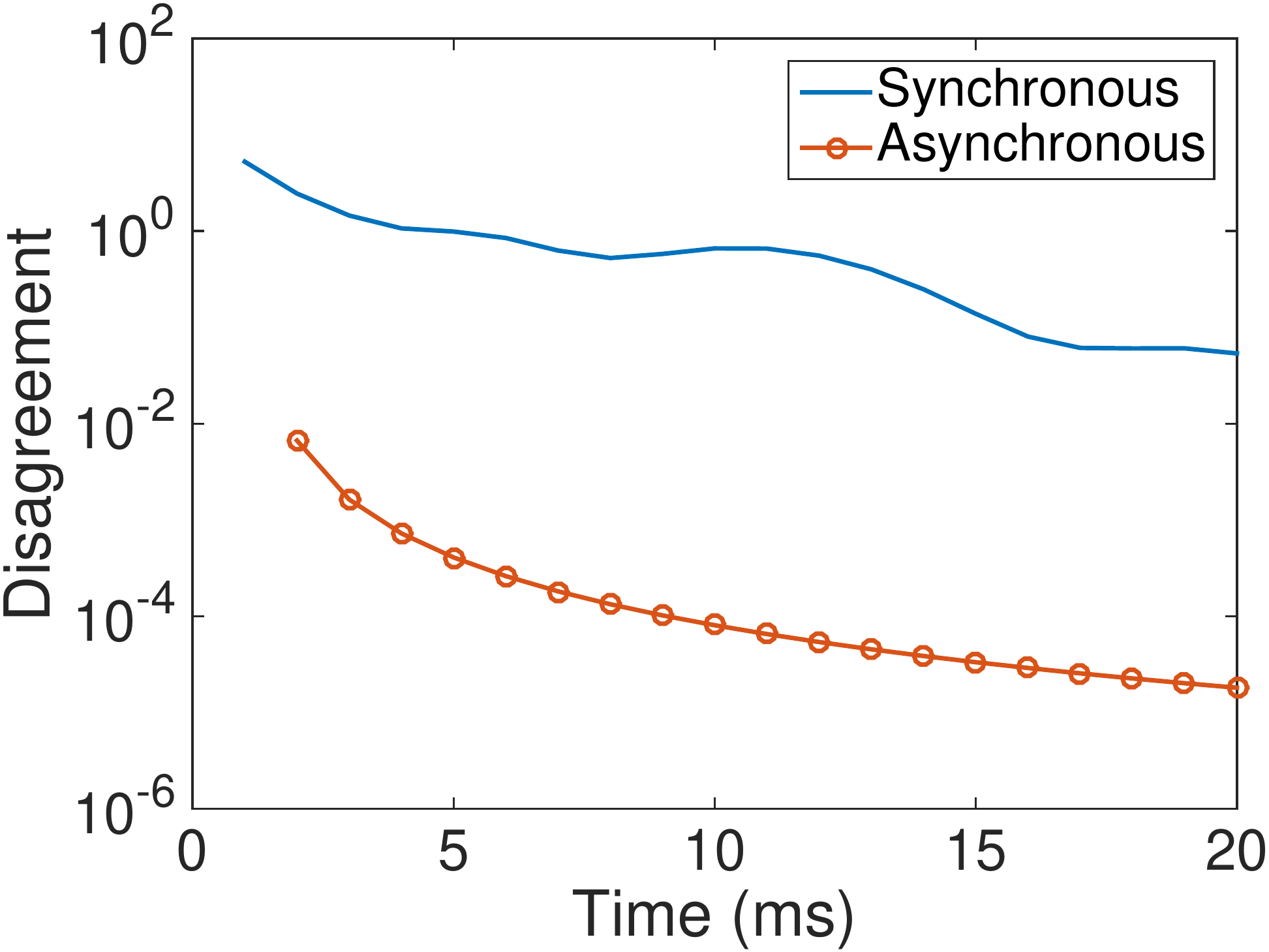}
\caption{Test on synthetic decentralized least-squares with and without delay/stochasticity (top)
and varying network size (bottom).
Left: objective function $f(z(T))$ versus
iteration number $T$. 
Right: disagreement $\sum_{i=1}^m \|y_i(T)-z(T)\|^2$ versus iteration number $T$.}
\label{fig:compare}
\end{figure}

\subsection{Test on real data}
We apply algorithm \eqref{eqn:algupd_node} to seismic tomography
where the data is collected and then processed by the nodes (sensors) 
in a wireless sensor network. In brief, underground seismic activities (such as earthquakes)
generate acoustic waves (we use P-wave here) which travel through
the materials and are detected by the sensors placed on the ground.
An explanatory picture of seismic tomography using a sensor network is shown in Figure \ref{fig:seismic_tomography}.
After data preprocessing, sensor $i$ obtains a matrix $A_i\in\mathbb{R}^{p_i\times n}$ 
and a vector $b_i\in\mathbb{R}^{p_i}$, and hence an objective $f_i(x)=(1/2)\|A_ix-b_i\|^2$
for $i=1,\dots,m$.
Here $(A_i)_{kl}$, the $(k,l)$-th entry of matrix $A_i$, is the distance that the wave generated by 
$k$-th seismic activity travels through pixel $l$, for $k=1,\dots,p_i$ ($p_i$ is the total number of 
seismic activities) and $l=1,\dots,n$ ($n$ is the total number of pixels
in the image), and $(b_i)_k$, the $k$-th component of $b_i$, is the total time that 
the wave travels from the source of $k$-th seismic activity to the sensor $i$.
Then $x_l$, the $l$-th component of $x\in\mathbb{R}^n$, represents the unknown ``slowness''
(reciprocal of the velocity of the traveling wave) at that location (pixel) $l$. The sensors then collaboratively solve for the image $x$ that minimizes the sum of their objective functions, under the constraint that only neighbor nodes may communicate during the computation process, since wireless signal transmission can only occur within a limited geographical range.
Once $x$ is reconstructed from $\min_xf(x)=\sum_{i=1}^mf_i(x)$, 
the material (e.g., rock, sand, oil, or magma) at each pixel $l$ can be identified by the value of $x_l$.
\begin{figure}[t!]
\centering
\includegraphics[width=.45\textwidth]{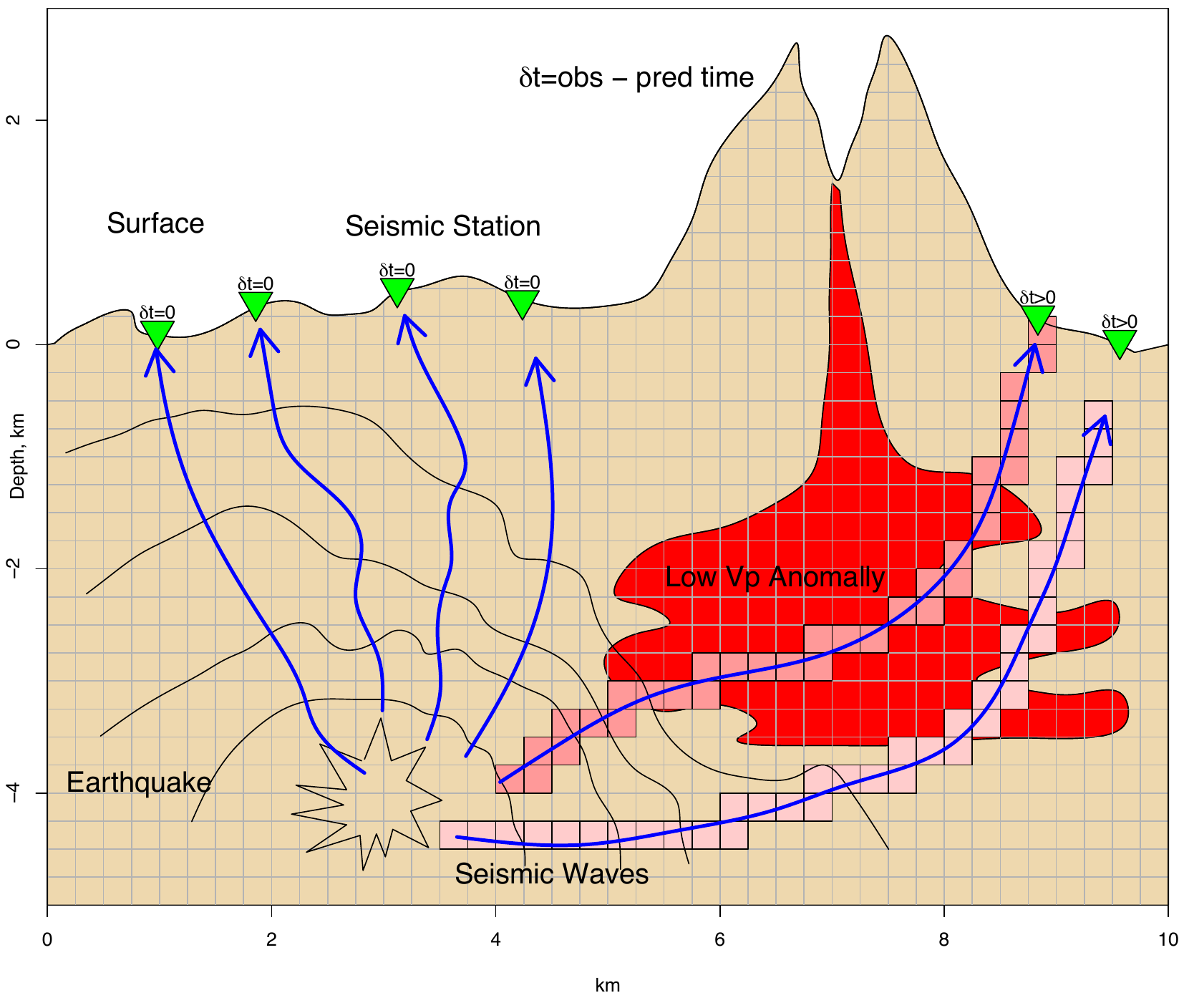}
\caption{Seismic tomography of an active volcano using wireless sensor network.
When there is a seismic activity (e.g., an earthquake) happens underground,
its acoustic waves (blue solid curves with arrows) travel to the ground surface and are detected by the
sensors (green triangles). Then the sensors communicate wirelessly
to reconstruct the entire image, where each square (tan, pink or red) represents
a pixel of the image $x\in\mathbb{R}^n$.}
\label{fig:seismic_tomography}
\end{figure}
%
%

The first dataset consists of a simple and connected network $G$ with $m=32$ nodes
where each node has $3$ neighbors, and
$A_i\in\mathbb{R}^{p_i\times n}$ and $b_i\in\mathbb{R}^{p_i}$
where the number of seismic events is $p_i=512$ and 
the size of a 2D image $x$ to be reconstructed is $n=64^2=4096$.
Since the matrix by stacking all $A_i$ is still underdetermined, we 
employ an objective function with Tikhonov regularization as
$f_i(x)=(1/2)(\|A_ix-b_i\|^2+\mu\|x\|^2)$ at each node $i$
where $\mu$ is set to $0.1$.
Note that more adaptive regularizers of $x$, such as $\ell_1$ and total variation (TV)
which result in a nonsmooth objective function, will be explored in future research.
We apply algorithm \eqref{eqn:algupd} with 
bound $B$ of delays set to $5$, $10$, and $20$ and
standard deviation $\sigma$ of stochastic gradient to $0.5$ and $0.05$.
We run our algorithm using step size $\alpha(t)=1/(2L+2\eta\sqrt{t})$ with 
$\eta$
that minimizes the constant of $1/\sqrt{T}$ term in Theorem \ref{thm:z_opt}.
The objective function $f(z(T))$ and disagreement 
$\sum_{i=1}^m \|y_i(T)-z(T)\|^2$ versus the iteration number $T$ are plotted in
the top row of Figure \ref{fig:real}, where convergence of both 
quantities can be observed.

The second seismic dataset contains a connected
network $G$ of size $m=50$ where each node has $3$ neighbors,
and matrices $A_i\in\mathbb{R}^{p_i\times n}$ and $b_i\in\mathbb{R}^{p_i}$
where $p_i=800$ and the size of 3D image $x$ to be reconstructed is $n=32^3=32768$.
We use the same objective function with Tikhonov regularization as before
with $\mu=0.01$. Other parameters are set the same as in the previous
test on a 2D seismic image. The settings for $B$ and $\sigma$ remain the same.
The objective function $f(z(T))$ and disagreement 
$\sum_{i=1}^m \|y_i(T)-z(T)\|^2$ versus the iteration number $T$ are plotted in
the middle row of Figure \ref{fig:real}, where similar convergence behavior can be observed.

The last seismic dataset consists of a connected
network $G$ of size $m=10$ where the average node degree is $5$,
and matrices $A_i\in\mathbb{R}^{p_i\times n}$ and $b_i\in\mathbb{R}^{p_i}$
where $p_i=1,816$ and the size of 3D image $x$ to be reconstructed is 
$n=160\times 200\times 24=768,000$.
In this test, we employ objective 
$f_i(x)=(1/2)(\|A_ix-b_i\|^2+\mu\|Dx\|^2)$ where 
$\mu=0.1$ and $D$ is the discrete gradient operator. 
Other parameters are set the same as in the previous
two seismic datasets. The bound $B$ of delay is set to $4$, $8$, and $16$, 
and standard deviation of stochastic gradient $\sigma$ is set to 1e-4 and
5e-4. 
The objective function $f(z(T))$ and disagreement 
$\sum_{i=1}^m \|y_i(T)-z(T)\|^2$ versus the iteration number $T$ are plotted in
the last row of Figure \ref{fig:real}.
The reconstructed image is displayed in the right panel of Figure \ref{fig:mshimage}.
By comparing with the solution obtained by centralized LSQR solver (left),
we can see the image is faithfully reconstructed on a decentralized network with
delayed stochastic gradients.

\begin{figure}[t]
\centering
\includegraphics[width=.4\textwidth]{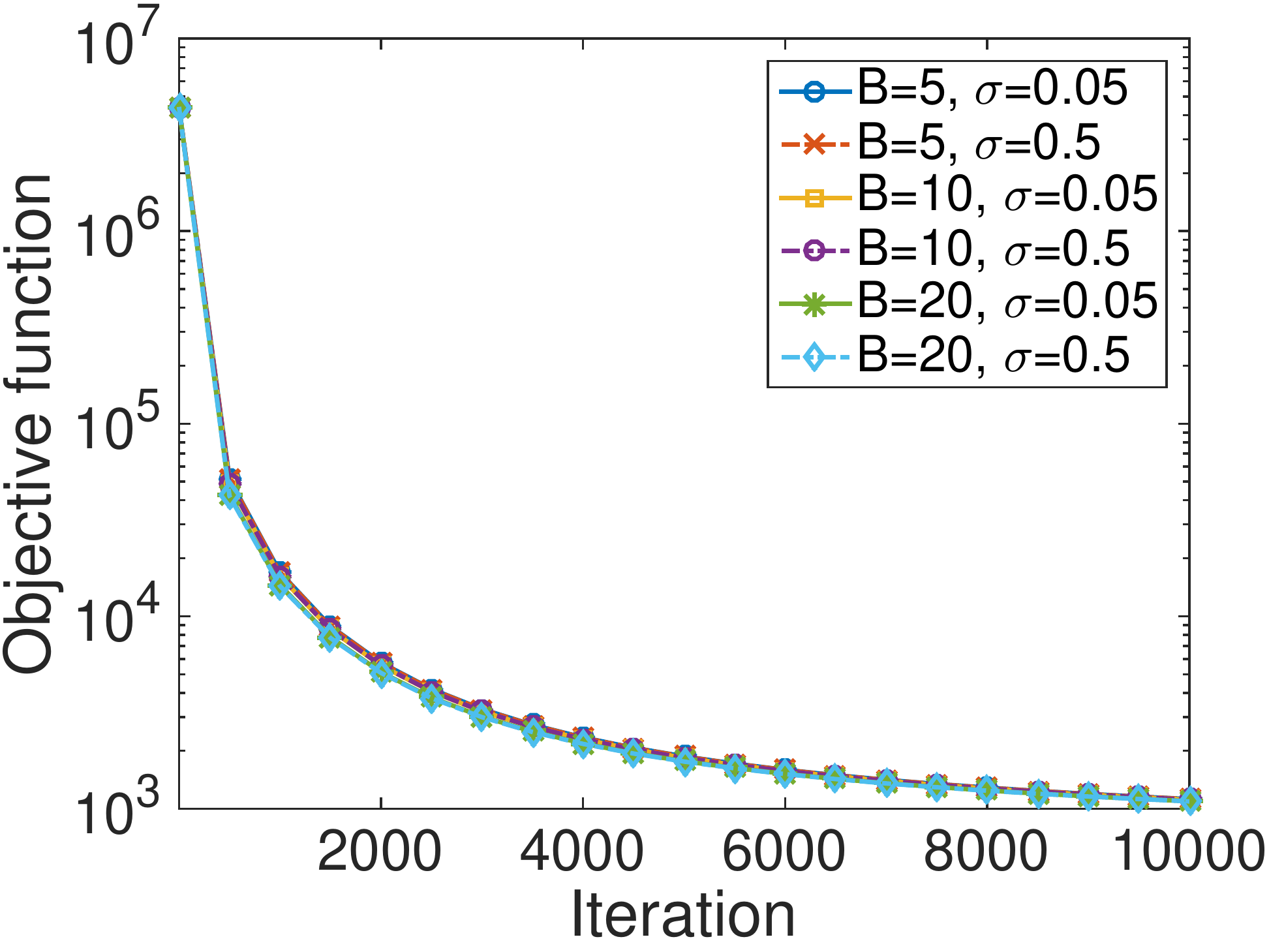}
\includegraphics[width=.4\textwidth]{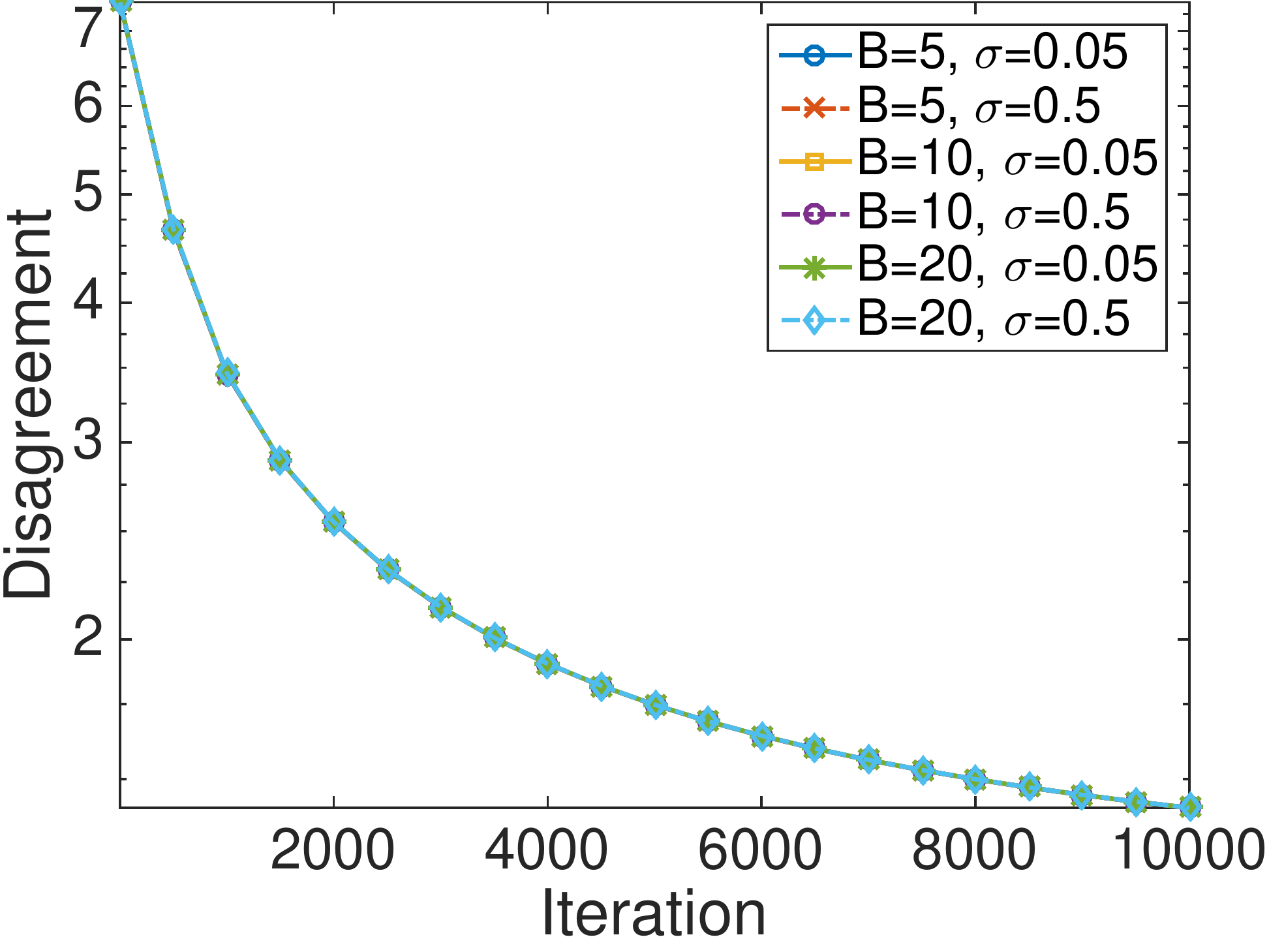}
\includegraphics[width=.4\textwidth]{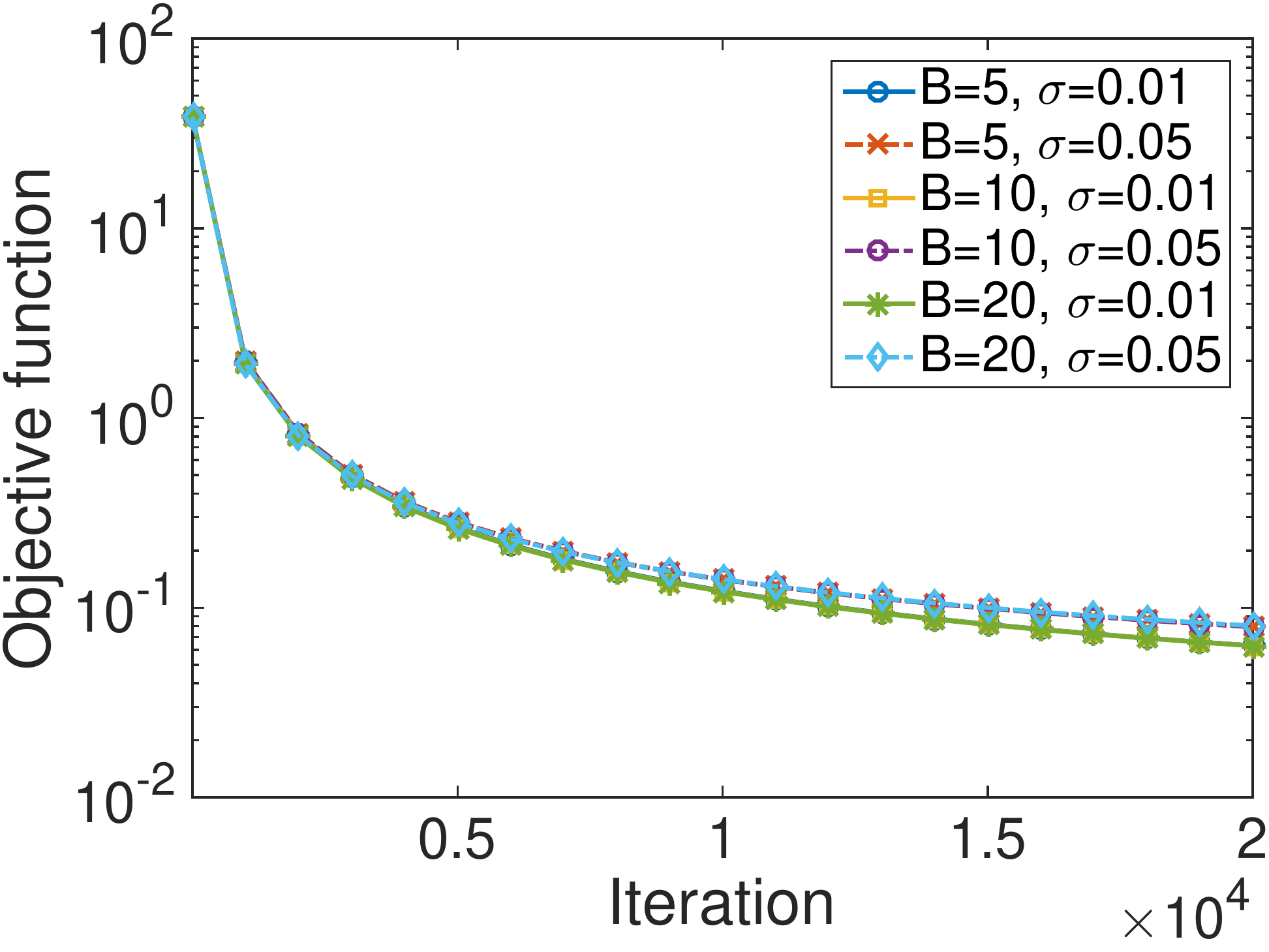}
\includegraphics[width=.4\textwidth]{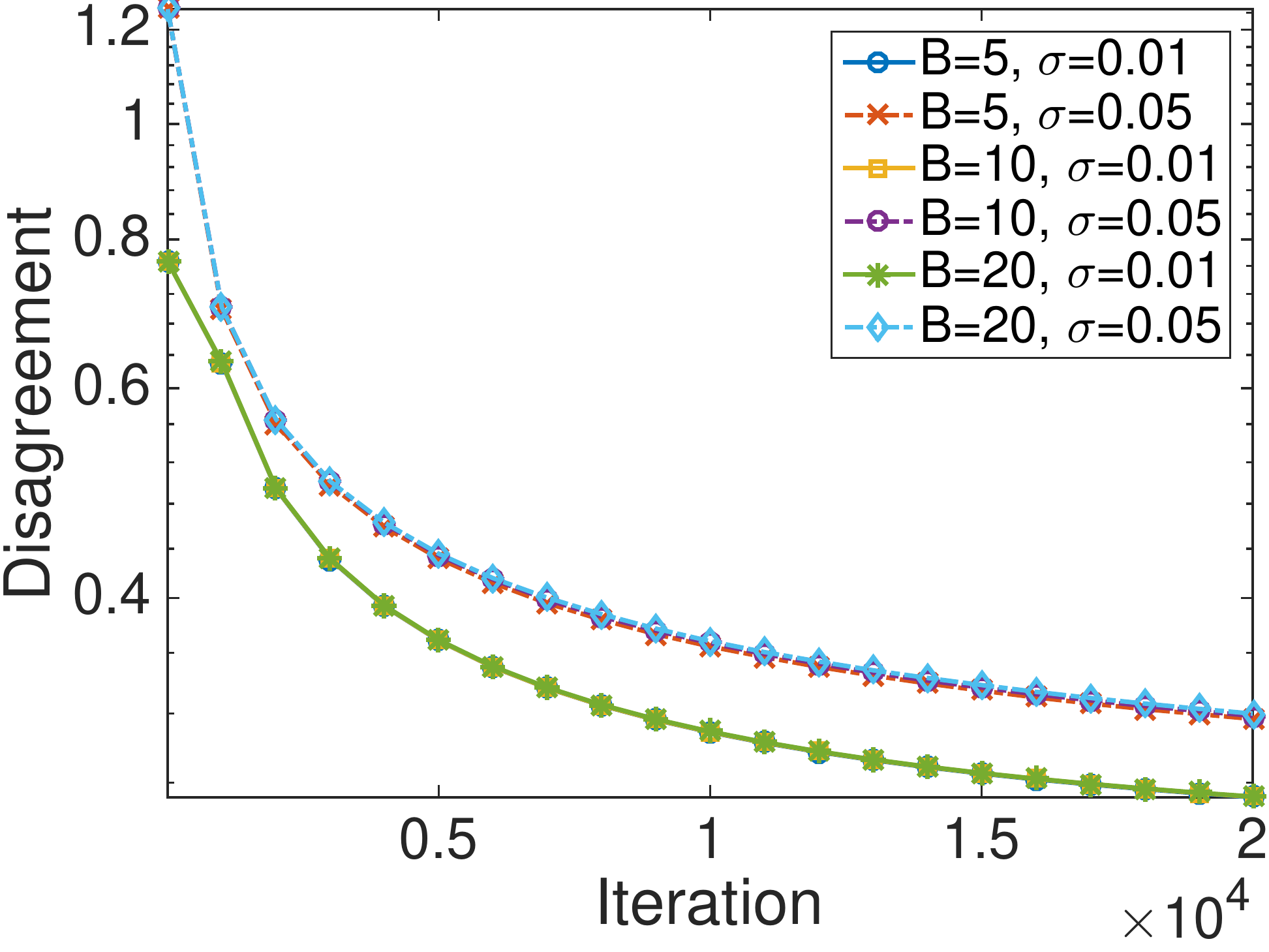}
\includegraphics[width=.4\textwidth]{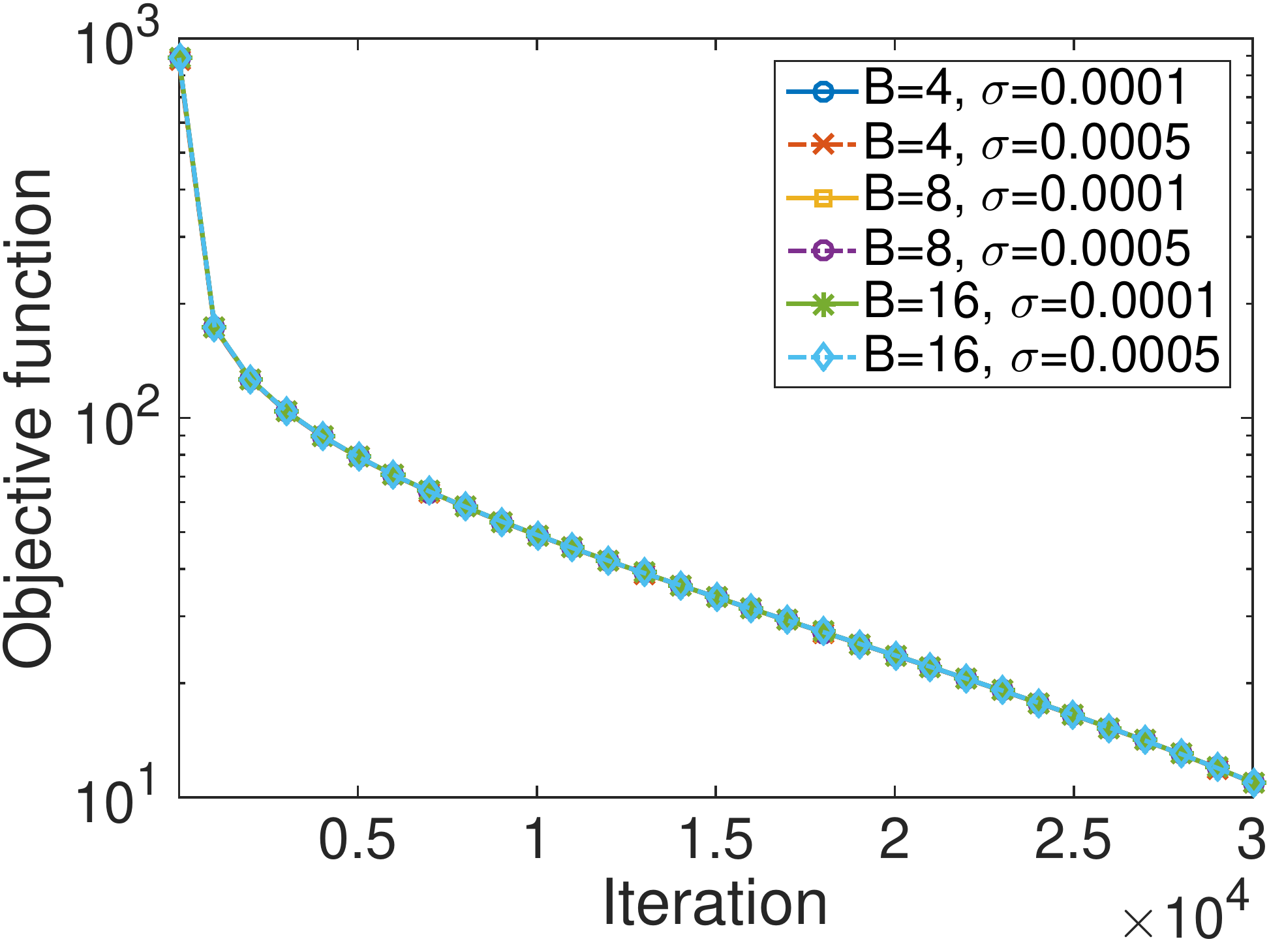}
\includegraphics[width=.4\textwidth]{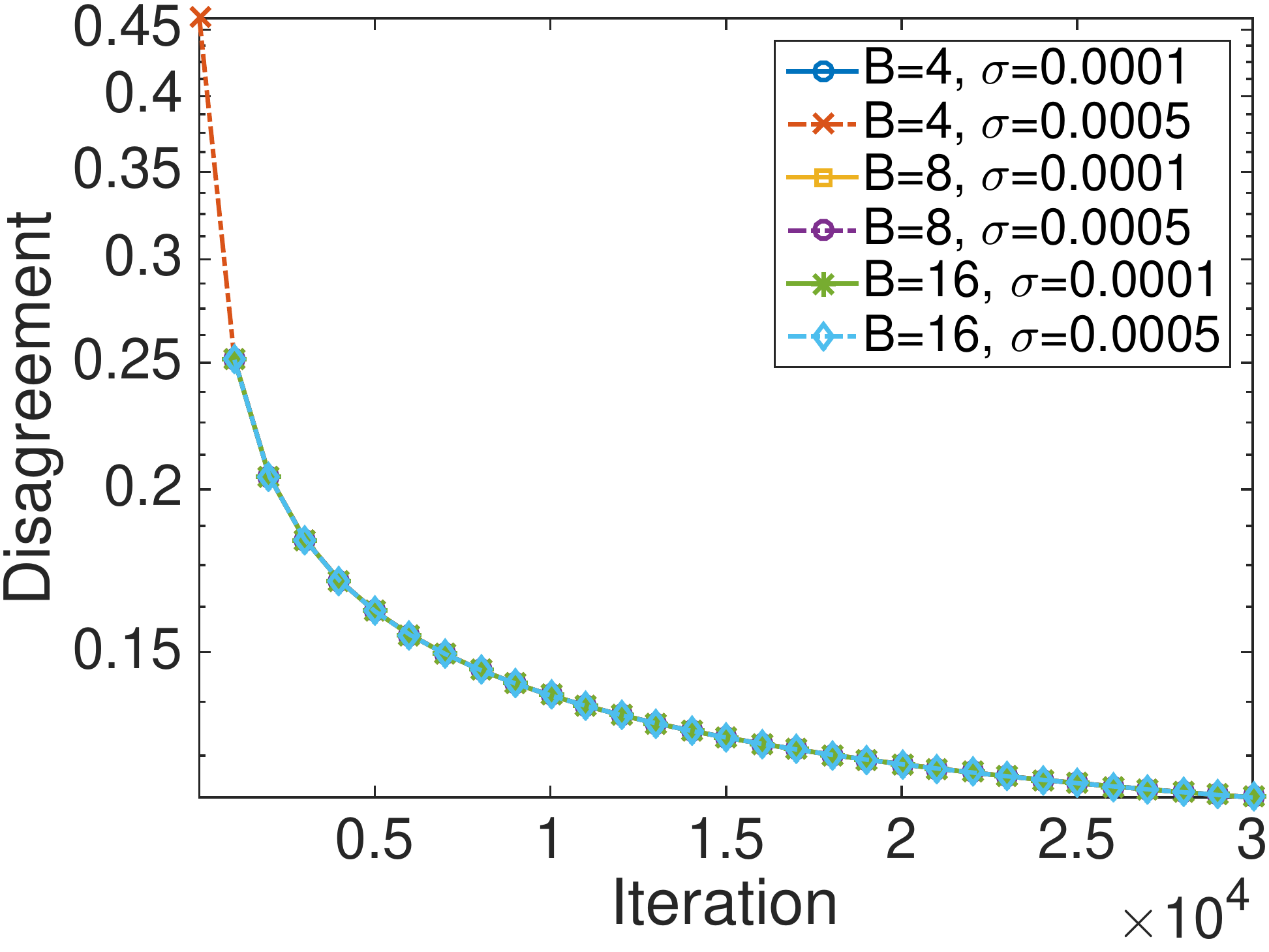}
\caption{Tests on real seismic image reconstruction problems
with $2D$ image with $n=64^2$ (top), $3D$ image with $n=32^3$ (middle),
and $3D$ image with $n=160\times 200\times 24$ (bottom) for different
levels of delay $B$ and standard deviation in stochastic
gradient $\sigma$. Left: objective function $f(z(T))$ versus
iteration number $T$. Optimal value indicates $f^*:=f(x^*)$. 
Right: disagreement $\sum_{i=1}^m \|y_i(T)-z(T)\|^2$ versus
iteration number $T$.}
\label{fig:real}
\end{figure}

\begin{figure}[t]
\centering
\includegraphics[width=.45\textwidth]{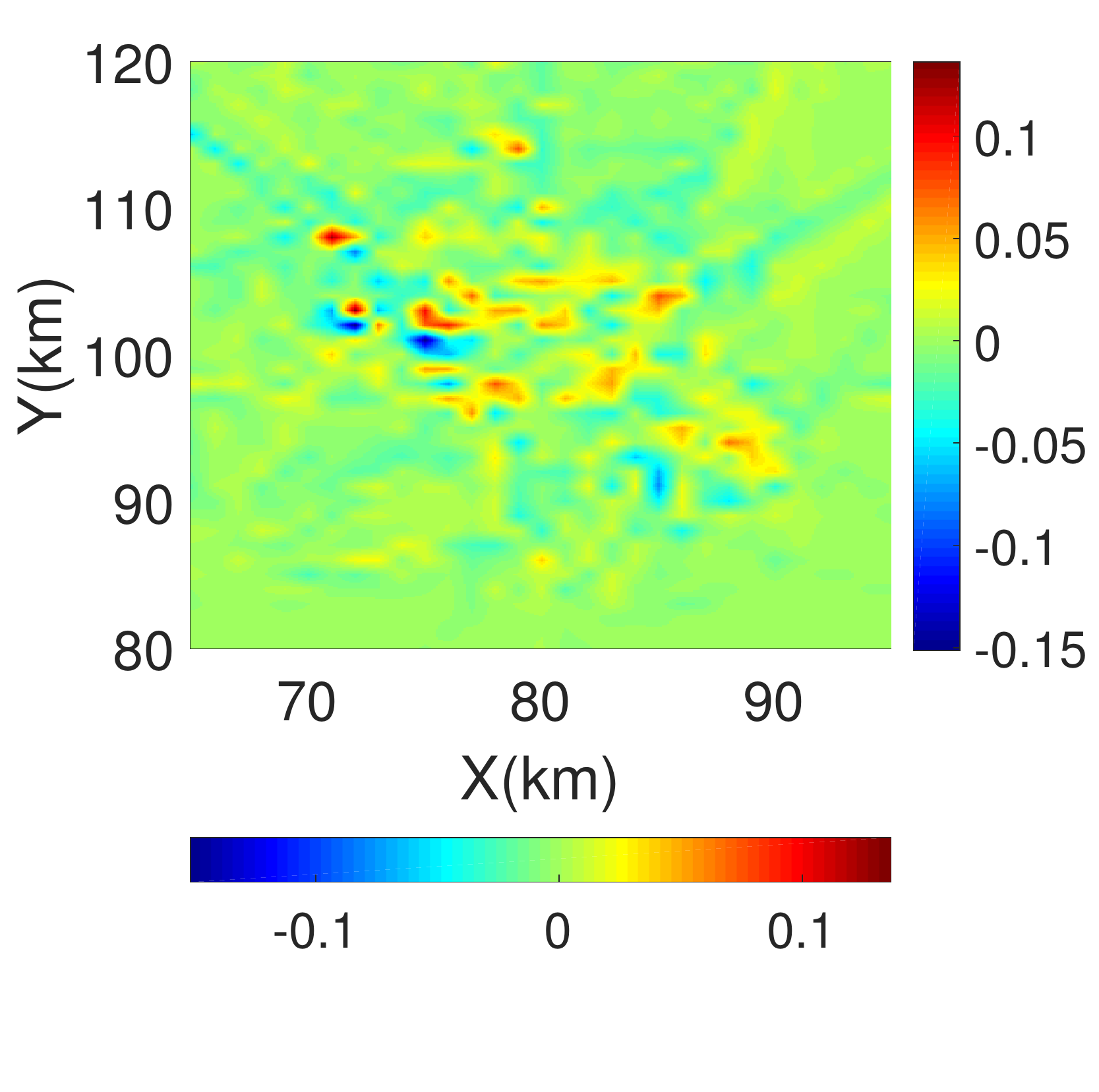}
\includegraphics[width=.45\textwidth]{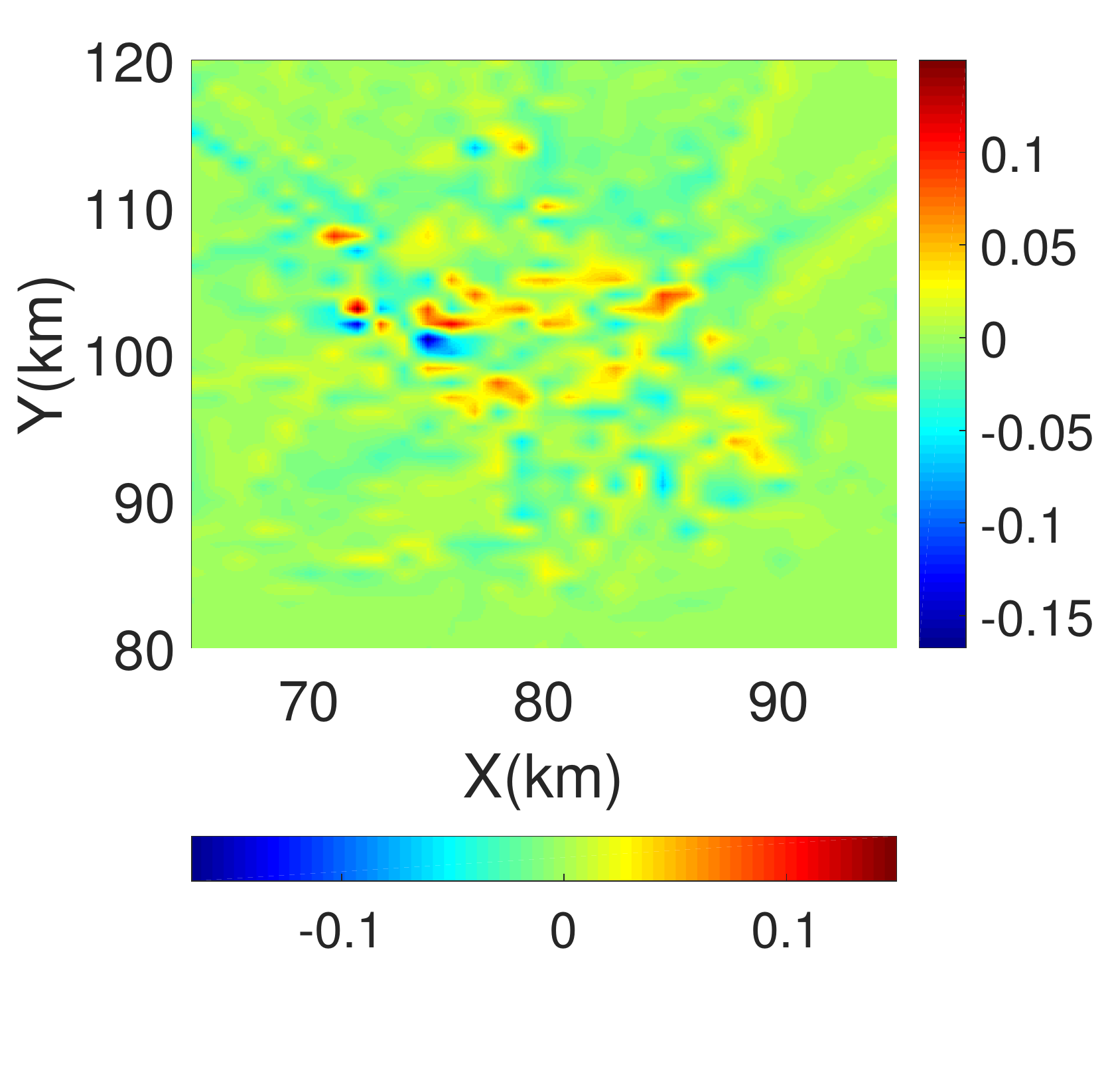}
\caption{Cross section of a reconstructed 3D seismic image generated by a centralized LSQR solver (left)
and decentralized algorithm with delayed stochastic gradient \eqref{eqn:algupd_node}
with $B=4$ and $\sigma=10^{-4}$ (right).}
\label{fig:mshimage}
\end{figure}

\section{Concluding Remarks}\label{sec:conclusion}
In this paper, we analyzed the convergence of decentralized delayed
stochastic gradient descent method as in \eqref{eqn:algupd_node}
for solving the consensus optimization \eqref{eqn:deccon}.
The algorithm takes into consideration that the nodes in the network 
privately hold parts of the objective 
function and collaboratively solve for the consensus
optimal solution
of the total objective while they can only communicate with their immediate neighbors,
as well as the delays of gradient information in real-world networks
where the nodes cannot be fully synchronized. 
We show that, as long as the random delays are bounded in expectation
and a proper diminishing step size policy is employed,
the iterates generated by the decentralized gradient decent method 
converge to a consensus solution.
Convergence rates of both objective and consensus were derived.
Numerical results on a number of synthetic and real data were also presented
for validation.

%
%
%
%
%

\appendix
\section{Proof of Lemma \ref{lemma:proj}}\label{app:proof_proj}
\begin{proof}
It suffices to show that for any fixed $R>0$ and $X=\{x\in\mathbb{R}^m:\|x\|_\infty\leq R\}$,
there is 
\begin{equation}
\|(I-J)\proj_X(x)\|\leq \|(I-J)x\|
\end{equation}
for all $x\in\mathbb{R}^m$.
Note that for $x=(x_1,x_2,\dots,x_m)^{\top}\in\mathbb{R}^m$, there is
$$\|(I-J)x\|^2=\sum_{i=1}^m(x_i-\overline{x})^2$$
where 
$\overline{x}:=(1/m)\sum_{i=1}^mx_i$. 
We only need to show that if all $\{x_i:x_i<-R\}$ are projected to $-R$ then $\|(I-J)x\|^2$ will reduce.
Without loss of generality, suppose $x_1,\dots,x_\ell<-R$ and $x_{\ell+1},\dots,x_m\ge-R$, and let 
denote the means of these two groups by
\begin{equation}
\mu_1:=\frac{1}{\ell}\sum_{i=1}^{\ell}x_i < -R 
\quad \mbox{ and }\quad \mu_2:=\frac{1}{m-\ell}\sum_{i=\ell+1}^mx_i\geq -R.
\end{equation}
Then we have $\overline{x}=(\ell\mu_1+(m-\ell)\mu_2)/m$, and 
\begin{align}&\quad \|(I-J)x\|^2 \nonumber \\
&=\sum_{i=1}^m(x_i-\overline{x})^2=\sum_{i=1}^m(x_i-\frac{\ell\mu_1+(m-\ell)\mu_2}{m})^2\nonumber\\
&=\sum_{i=1}^{\ell}(x_i-\frac{\ell\mu_1+(m-\ell)\mu_2}{m})^2
+\sum_{i=\ell+1}^m(x_i-\frac{\ell\mu_1+(m-\ell)\mu_2}{m})^2\nonumber\\
&=\sum_{i=1}^\ell\left((x_i-\mu_1)+\frac{m-\ell}{m}(\mu_1-\mu_2)\right)^2
+\sum_{i=\ell+1}^m\left((x_i-\mu_2)+\frac{\ell}{m}(\mu_2-\mu_1)\right)^2 \label{eqn:sumfoil} \\
&=\sum_{i=1}^{\ell}(x_i-\mu_1)^2+2\frac{m-\ell}{m}(\mu_1-\mu_2)\sum_{i=1}^{\ell}(x_i-\mu_1)
+\ell\left(\frac{m-\ell}{m}\right)^2(\mu_1-\mu_2)^2 \nonumber \\
\;\;\;\;&\ +\sum_{i=\ell+1}^m(x_i-\mu_2)^2+2\frac{\ell}{m}(\mu_2-\mu_1)\sum_{i=\ell+1}^m(x_i-\mu_2)
+(m-\ell)\left(\frac{\ell}{m}\right)^2(\mu_2-\mu_1)^2 \nonumber
\end{align}
After $x_1,\cdots,x_{\ell}$ are projected to $-R$ 
(and $x_{\ell+1},\dots,x_m$ remain unchanged), their mean is updated from $\mu_1$ to $-R$
for all $i=1,\dots,\ell$, and $\mu_2-\mu_1(\geq0)$ reduces to $\mu_2+R(\geq0)$.
Therefore, the first, third, and sixth terms in the right hand side of \eqref{eqn:sumfoil} are decreased, 
the second and fifth terms remain zero, and the fourth term remains unchanged. 
Thus $\|(I-J)x\|$ reduces after projection to 
$[-R,\infty)^m$. 
A similar argument implies that projecting $\{x_i:x_i>R\}$ to $R$ will further reduce $\|(I-J)x\|^2$. 
Therefore projecting $x$ to $X$, i.e., projecting to $[-R,\infty)^m$ and then $(-\infty,R]^m$, reduces $\|(I-J)x\|^2$.
\end{proof}

\section{Proof of Lemma \ref{lemma:seqrate}}\label{app:proof_seqrate}
\begin{proof}
First, we note that 
\begin{equation}\label{eqn:sum_alpha}
\sum_{s=0}^{t-1}\alpha(s)\lambda^{t-1-s}=\alpha(0)\lambda^{t-1}+\alpha(1)\lambda^{t-2}+\sum_{s=2}^{t-1}\alpha(s)\lambda^{t-1-s}
\end{equation}
which means that the rate is upper bounded by the last sum on the right side above since the 
first two tend to $0$ at a linear rate $\lambda\in(0,1)$.

Note that for all $w\in[s-1,s]$ we have 
$\frac{1}{\sqrt{s}}\leq\frac{1}{\sqrt{w}}$ and $\lambda^{-s}\leq\lambda^{-(w+1)}$ since $\lambda\in(0,1)$, and therefore
\begin{equation}
\alpha(s)\lambda^{t-1-s}=\frac{\lambda^{t-1-s}}{c_1+c_2\sqrt{s}}\leq\frac{\lambda^{t-1}\lambda^{-s}}{c_2\sqrt{s}}\leq\frac{\lambda^{t-1}\lambda^{-(w+1)}}{c_2\sqrt{w}}=\frac{\lambda^{t-2-w}}{c_2\sqrt{w}}.
\end{equation}
This inequality allows us to bound the last term on right hand side of \eqref{eqn:sum_alpha} by
\begin{equation}\label{eqn:dominating_sum_alpha}
\sum_{s=2}^{t-1}\alpha(s)\lambda^{t-1-s}\leq\sum_{s=2}^{t-1}\int_{s-1}^s\frac{\lambda^{t-2-w}}{c_2\sqrt{w}}dw
=\int_{1}^{t-1}\frac{\lambda^{t-2-w}}{c_2\sqrt{w}}dw=\frac{2\lambda^{t-2}}{c_2} I_t,
\end{equation}
where $I_t$ is defined by
\begin{align}
I_t:=\frac{1}{2}\int_1^{t-1}\frac{\lambda^{-w}}{\sqrt{w}}dw.
\end{align}
By changing of variable $w=u^2$, we obtain $I_t=\int_1^{\sqrt{t-1}}\lambda^{-u^2}du$. Now we have that
%
\begin{align}
I_t^2 & = \int_1^{\sqrt{t-1}}\int_1^{\sqrt{t-1}}\lambda^{-(u^2+v^2)}\ dudv
 = \int_1^{\sqrt{t-1}}\int_1^{\sqrt{t-1}} e^{-(u^2+v^2)\log\lambda}\ dudv \nonumber \\
& \leq \int_0^{\sqrt{t}}\int_0^{\sqrt{t}}e^{-(u^2+v^2)\log\lambda}\ dudv 
 = 2\int_0^{\pi/4}\int_0^{\sqrt{t}/\cos\theta}e^{-\rho^2\log\lambda}\rho \ d\rho d\theta \label{eqn:doubleint} \\
& =-\frac{1}{\log\lambda}\int_0^{\pi/4}(e^{-t\log\lambda/\cos^2(\theta)}-1)\ d\theta 
 < -\frac{1}{\log\lambda}\int_0^{\pi/4}e^{-t\log\lambda/\cos^2(\theta)}\ d\theta \nonumber
\end{align}
where the third equality comes from changing to a polar system with the substitutions 
$u=\rho\cos\theta$ and $v=\rho\sin\theta$.
Note that $\cos^{-2}(\theta)-( 1+4\theta/\pi)\leq0$ for all 
$\theta\in[0,\pi/4]$ since $\cos^{-2}(\theta)-1-4\theta/\pi$ is convex with respect to $\theta$ and vanishes
at $\theta=0$ and $\theta=\pi/4$. 
Therefore 
\begin{align}
I_t^2 \leq -\frac{1}{\log\lambda}\int_0^{\pi/4}e^{-t\log\lambda(1+4\theta/\pi)}d\theta
\leq \frac{\pi\lambda^{-2t}}{4t(\log\lambda)^2}.
\end{align}
Hence the sum in \eqref{eqn:dominating_sum_alpha} is bounded by
\begin{equation}
\sum_{s=2}^{t-1}\alpha(s)\lambda^{t-1-s}\leq\frac{2\lambda^{t-2}}{c_2}I_t
\leq \frac{2\lambda^{t-2}}{c_2}\frac{\sqrt{\pi}\lambda^{-t}}{2\sqrt{t}\log(\lambda^{-1})}
=\frac{\sqrt{\pi}\lambda^{-2}}{c_2\sqrt{t}\log(\lambda^{-1})}
\end{equation}
which completes the proof.
\end{proof}

\section{Proof of Corollary \ref{cor:adjrate}}\label{app:proof_adjrate}
\begin{proof}
According to the update \eqref{eqn:alg} or equivalently \eqref{eqn:algupd}, we have
\begin{align}\label{eqn:adjbound}
\Ex[\|x(t+1)-x(t)\|]
& = \Ex[\|\proj_{\calX} [Wx(t)-\alpha(t)g(t-\tau(t))] - x(t)\|] \nonumber \\
& \leq \Ex[\|(I-W)x(t)+\alpha(t)g(t-\tau(t)))\|] \\
& \leq \Ex[\|(I-W)x(t)\|]+\alpha(t)\Ex[\|g(t-\tau(t)))\|] \nonumber
\end{align}
where we used the facts that $x(t)\in \calX$ and that projection $\proj_{\calX}$ is non-expansive
in the first inequality.
Note that $WJ=J$ and hence $I-W=(I-W)(I-J)$, we have 
\begin{align*}
\Ex[\|(I-W)x(t)\|] = \Ex[\|(I-W)(I-J)x(t)\|] \leq \Ex[\|(I-J)x(t)\|] \leq \frac{\sqrt{\pi m}G\lambda^{-2}}{\eta\sqrt{t}\log(\lambda^{-1})}
\end{align*}
where we used the fact that $\|I-W\|\leq 1$ in the first inequality and applied Theorem \ref{thm:cssrate}
to obtain the second inequality. 

Furthermore, we have by the definition of $\alpha(t)$ that
\begin{align}\label{eqn:alphaGrate}
\|\alpha(t)g(t-\tau(t))\| \leq \sqrt{m}\alpha(t)G=\frac{\sqrt{m}G}{2(L+\eta\sqrt{t})}\leq \frac{\sqrt{m}G}{2\eta\sqrt{t}}.
\end{align}
Applying the two inequalities above to \eqref{eqn:adjbound}
yields \eqref{eqn:adjrate}.
\end{proof}

\section{Proof of Corollary \ref{cor:delayrate}}\label{app:proof_delayrate}
\begin{proof}
We first define $\taubar(t):=\max\{\tau_i(t):1\leq i\leq m\}$. 
{Then there is $\Ex[|\taubar(t)|^2]
\leq \Ex[\sum_{i=1}^m|\tau_i(t)|^2]\leq m B^2$.} 
Without loss of generality, we assume that $0\leq \taubar(t)\leq t-2$
for every given $t$, i.e., we consider the convergence rate when every node 
has successfully computed their own gradient at least twice. Then we obtain that
\begin{align}
& \quad\, \Ex[\|x(t)-x(t-\tau(t))\|] \nonumber \\
& \leq \Ex\sbr[3]{ \sum_{s=1}^{\taubar(t)} \|x(t-s+1)-x(t-s)\| } \leq C\Ex\sbr[3]{ \sum_{s=1}^{\taubar(t)} \frac{1}{\sqrt{t-s}} } \nonumber \\
& = C\Ex\sbr[3]{ \sum_{s=t-\taubar(t)}^{t-1} \frac{1}{\sqrt{s}} }\leq C\Ex\sbr[3]{\int_{t-\taubar(t)-1}^{t-1} \frac{1}{\sqrt{s}}ds} \label{eq:gap_x_tau} \\
& = 2C\Ex\sbr{ \sqrt{t-1}-\sqrt{t-\taubar(t)-1} }\leq 2C\Ex\sbr[3]{\frac{\taubar(t)}{\sqrt{t-1}+\sqrt{t-\taubar(t)-1}}} \nonumber \\
& \leq C\Ex\sbr[3]{\frac{\taubar(t)}{\sqrt{t-\taubar(t)-1}}} \nonumber
\end{align}
where we used triangle inequality to obtain the first inequality, applied Corollary \ref{cor:adjrate} to obtain the
second inequality, and used the fact that $\taubar(t)\geq0$ to obtain the last inequality above.
Note that there is
\begin{align}
\Ex\sbr{\frac{\taubar(t)}{\sqrt{t-\taubar(t)-1}}}
&=\sum_{s=0}^{\lfloor t/2\rfloor-1}\frac{s}{\sqrt{t-s-1}}\prob(\taubar(t)=s)+ \sum_{s=\lfloor t/2\rfloor}^{t-2}\frac{s}{\sqrt{t-s-1}}\prob(\taubar(t)=s) \nonumber\\
& \leq\frac{\sqrt{2}}{\sqrt{t}}\sum_{s< t/2}s\prob(\taubar(t)=s)+(t-2)\sum_{s\geq t/2}\prob(\taubar(t)=s) \label{eq:ex_frac}\\
& \leq\frac{\sqrt{2m}B}{\sqrt{t}}+\frac{4mB^2(t-2)}{t^2} \leq\frac{\sqrt{2m}B}{\sqrt{t}}+\frac{4mB^2}{t}= O\del{\frac{1}{\sqrt{t}}} \nonumber
\end{align}
where we used the fact that $\sqrt{t-s-1}\geq \sqrt{t/2}$ if $0\leq s\leq \lfloor t/2 \rfloor-1$
and $s/\sqrt{t-s-1}\leq t-2$ if $\lfloor t/2\rfloor\leq s\leq t-2$ to obtain the first
inequality, and $\sum_{s< t/2} s \prob(\taubar(t)=s)\leq \Ex[\taubar(t)]\leq \sqrt{\Ex[\taubar(t)^2]}=\sqrt{m}B$
and $\sum_{s\geq t/2}\prob(\taubar(t)=s)=\prob(\taubar(t)\geq t/2)\leq (4/t^2)\Ex[\taubar(t)^2]\leq 4mB^2/t^2$
(by Chebyshev's inequality) in the second inequality.
In particular, it is easy to verify that, when $t\geq 8mB^2$, there is 
$\sqrt{2m}B/\sqrt{t}\geq 4mB^2/t$
and hence $\Ex\sbr[2]{\frac{\taubar(t)}{\sqrt{t-\taubar(t)-1}}}\leq \frac{2\sqrt{2m}B}{\sqrt{t}}$.
Combining \eqref{eq:gap_x_tau} and \eqref{eq:ex_frac} completes the proof.
\end{proof}

\section{Proof of Lemma \ref{lemma:sum_inprod}}\label{app:proof_sum_inprod}
\begin{proof}
By Cauchy-Schwarz inequality, we have that
\begin{align*}
& \quad\, \sum_{t=1}^T \langle \nabla f(x(t))-\nabla f(x(t-\tau(t))),x(t+1)-x^*\rangle \\
& \leq \sum_{t=1}^T \| \nabla f(x(t))-\nabla f(x(t-\tau(t))\| \| x(t+1)-x^*\|
\end{align*}
Note that $\|x(t+1)-x^*\|^2=\sum_{i=1}^{m} \| x_i(t+1)-x^*\|^2 \leq mn(2R)^2$ due 
to the bound of $X=\{x\in\mathbb{R}^n:\|x\|_\infty\leq R\}$,
and $\| \nabla f(x(t))-\nabla f(x(t-\tau(t))\|^2 = \sum_{i=1}^m \| \nabla f_i(x_i(t))-\nabla f_i(x_i(t-\tau(t))\|^2
\leq \sum_{i=1}^m L_i\| x_i(t)-x_i(t-\tau(t))\|^2 \leq L \|x(t)-x(t-\tau(t))\|^2\leq 2\sqrt{2m}CB/\sqrt{t}$
due to Corollary \ref{cor:delayrate}. Therefore, we obtain
\begin{align*}
\sum_{t=1}^T \langle \nabla f(x(t))-\nabla f(x(t-\tau(t))),x(t+1)-x^*\rangle
\leq 8\sqrt{2nLT}mRCB  \end{align*}
by using the fact that $\sum_{t=1}^T 1/\sqrt{t}\leq 2\sqrt{T}$. 
This completes the proof.
\end{proof}

\section{Proof of Theorem \ref{thm:y_opt}}\label{app:proof_y_opt}
\begin{proof}
We first note that there is
\begin{align}\label{eqn:objdiff}
& \qquad f(x(t+1))-f(x^*) = \sum_{i=1}^m \left(f_i(x_{i}(t+1))-f_i(x^*) \right) \nonumber\\
& = \ \sum_{i=1}^m \left[f_i(x_{i}(t+1))-f_i(x_{i}(t))+f_i(x_{i}(t))-f_i(x^*) \right] \nonumber\\
& \leq \ \sum_{i=1}^m \left[\left\langle \nabla f_i(x_i(t)),x_{i}(t+1)-x_i(t) \right\rangle+\frac{L_i}{2}\|x_{i}(t+1)-x_i(t)\|^2  \right. \\
&\qquad\qquad \left.+\left\langle \nabla f_i(x_i(t)),x_{i}(t)-x^* \right\rangle \right] \nonumber\\
& \leq \ \sum_{i=1}^m \left[\left\langle \nabla f_i(x_i(t)),x_{i}(t+1)-x^* \right\rangle + 
\frac{L_i}{2}\|x_{i}(t+1)-x_i(t)\|^2 \right] \nonumber \\
& \leq \ \left\langle \nabla f(x(t)),x(t+1)-x^* \right\rangle + 
\frac{L}{2}\|x(t+1)-x(t)\|^2  \nonumber \\
& \leq \ \left\langle g(t-\tau(t)),x(t+1)-x^* \right\rangle + \left\langle \nabla f(x(t))-g(t-\tau(t)),x(t+1)-x^* \right\rangle \nonumber \\
& \qquad\qquad +\frac{L}{2}\|x(t+1)-x(t)\|^2  \nonumber
\end{align}
%
where we used the $L_i$-Lipschitz continuity of $\nabla f_i$ and
convexity of $f_i$ to obtain the first inequality.
Note that $x(t+1)$ is obtained by \eqref{eqn:alg} as
\begin{align}
x(t+1)= & \argmin_{x \in \calX} \left\{ \langle g(t-\tau(t)),x\rangle + \frac{1}{2\alpha(t)}\|x-Wx(t)\|^2\right\} \\
=& \argmin_{x \in \calX} \left\{ \left\langle g(t-\tau(t))+\frac{1}{\alpha(t)}(I-W)x(t),x\right\rangle + \frac{1}{2\alpha(t)}\|x-x(t)\|^2
\right\}\nonumber 
\end{align}
Therefore, the optimality of $x(t+1)$ in \eqref{eqn:alg} and strong convexity of the objective function
in \eqref{eqn:alg} imply that
\begin{align}\label{eqn:inner_opt}
&\left\langle g(t-\tau(t)),x(t+1)-x^* \right\rangle \nonumber \\
\leq &-\frac{1}{\alpha(t)}\left\langle (I-W)x(t),x(t+1)-x^* \right\rangle \\
& +\frac{1}{2\alpha(t)} \left[\|x^*-x(t)\|^2-\|x(t+1)-x(t)\|^2-\|x^*-x(t+1)\|^2\right]. \nonumber
\end{align}
Furthermore, we note that $\bfone\in\mathrm{Null}(I-W)$ and $x^*$ is consensual, hence we have
\begin{align}\label{eqn:cons_inner_prod}
&\ -\frac{1}{\alpha(t)}\langle (I-W)x(t),x(t+1)-x^*\rangle \nonumber\\
= &\  -\frac{1}{\alpha(t)}\langle (I-W)(x(t)-x^*),x(t+1)-x^*\rangle \nonumber\\
= &\ \frac{1}{2\alpha(t)}\left(\|x(t)-x(t+1)\|_{I-W}^2-\|x(t)-x^*\|_{I-W}^2-\|x(t+1)-x^*\|_{I-W}^2\right) \\
\leq & \ \frac{1}{4\alpha(t)}\|x(t)-x(t+1)\|_{I-W}^2\nonumber
\end{align}
where we have used the fact that
\[\| x(t)-x(t+1)\|_{I-W}^2 \leq 2(\|x(t)-x^*\|_{I-W}^2+\|x(t+1)-x^*\|_{I-W}^2 )\]
to obtain the inequality above. We also have that
\[\|x(t)-x(t+1)\|_{I-W}^2\leq \|x(t)-x(t+1)\|^2\]
with which we can further bound \eqref{eqn:cons_inner_prod} as
\begin{align*}
-\frac{1}{\alpha(t)} \langle (I-W)x(t),x(t+1)-x^*\rangle 
\leq \frac{1}{4\alpha(t)}\|x(t)-x(t+1)\|^2.
\end{align*}
Now applying the inequality above and \eqref{eqn:inner_opt} to \eqref{eqn:objdiff},
and taking sum of $t$ from $1$ to $T$, we get
\begin{align}\label{eqn:sumobjdiff}
\sum_{t=1}^{T} f(x(t+1))- T f(x^*) & \leq \ \sum_{t=1}^{T}  \frac{1}{2\alpha(t)}\left(\|x(t)-x^*\|^2-\|x(t+1)-x^*\|^2\right)  \nonumber \\
& \quad +\sum_{t=1}^{T} \left(\frac{L}{2}-\frac{1}{4\alpha(t)}\right)\|x(t)-x(t+1)\|^2  \\
& \quad + \sum_{t=1}^{T} \left\langle \nabla f(x(t))- g(t-\tau(t)), x(t+1)-x^*\right\rangle. \nonumber
\end{align}



Note that the running average $y(T)=(1/T)\sum_{t=1}^{T} x(t+1)$ satisfies
$f(y(T))\leq (1/T)\sum_{t=1}^{T} f(x(t+1))$ due to the convexity of all $f_i$.
Therefore, together with \eqref{eqn:sumobjdiff} 
and the definition of $\alpha(t)$, we have
\begin{align}\label{eqn:sumobj}
& \ T [f(y(T))- f(x^*)] \nonumber \\
\leq & \ \sum_{t=1}^{T}  \left[\frac{1}{2\alpha(t)} \left(\|x(t)-x^*\|^2-\|x(t+1)-x^*\|^2\right)
-\frac{\eta\sqrt{t}}{2}\|x(t)-x(t+1)\|^2 \right]\\
& \ + \sum_{t=1}^{T} \left\langle \nabla f(x(t))- g(x(t-\tau(t))), x(t+1)-x^*\right\rangle. \nonumber
\end{align}

Now, by taking expectation on both sides of \eqref{eqn:sumobj}, we obtain
\begin{align}\label{eqn:Esumobj}
T\Ex[f(y(T))- f(x^*)] \leq & \ \sum_{t=1}^{T}  \left[\frac{1}{2\alpha(t)} \left(e(t)-e(t+1)\right)
-\frac{\eta\sqrt{t}}{2}\Ex[\|x(t)-x(t+1)\|^2] \right] \nonumber\\
&+  8\sqrt{2nLT} mRCB \\
&+ \sum_{t=1}^{T} \Ex\left\langle \nabla f(x(t-\tau(t)))- g(t-\tau(t)), x(t+1)-x^*\right\rangle \nonumber
\end{align}
where we denoted $e(t):=\Ex[\|x(t)-x^*\|^2]$ for notation simplicity.

Now we work on the last sum of inner products on the right side of \eqref{eqn:Esumobj}.
First we observe that
\begin{align}\label{eqn:sumobjdiff_inprod_rhs2_tmp}
&\ \Ex \left\langle \nabla f(x(t-\tau(t)))- g(t-\tau(t)), x(t+1)-x^*\right\rangle \nonumber\\
=  &\  \Ex \left\langle \nabla f(x(t-\tau(t)))- g(t-\tau(t)), x(t-\tau(t))-x^*\right\rangle \\
&\qquad + \Ex \left\langle \nabla f(x(t-\tau(t)))- g(t-\tau(t)), x(t+1)-x(t-\tau(t))\right\rangle. \nonumber
\end{align}
Note that $g_i(t-\tau_i(t))$ is the stochastic gradient of node $i$ evaluated at iteration $t-\tau_i(t)$, 
and the stochastic error $g_i(t-\tau_i(t))-\nabla f_i(x_i(t-\tau_i(t)))$ 
is independent of $x_i(t-\tau_i(t))$. Therefore, we have
\begin{align}\label{eqn:Einprod1}
& \quad\, \Ex \left\langle \nabla f(x(t-\tau(t)))- g(t-\tau(t)), x(t-\tau(t))-x^*\right\rangle \\
& = \sum_{i=1}^m \Ex \left\langle \nabla f_i(x_i(t-\tau_i(t)))- g_i(t-\tau_i(t)), x_i(t-\tau_i(t))-x^*\right\rangle=0, \nonumber
\end{align}
since the stochastic gradients are unbiased.
Furthermore, by Young's inequality, we have
\begin{align}\label{eqn:Einprod2}
&\ \Ex \left\langle \nabla f(x(t-\tau(t)))- g(t-\tau(t)), x(t+1)-x(t-\tau(t))\right\rangle \nonumber \\
\leq &\  \frac{2}{\eta\sqrt{t}}\Ex [\| \nabla f(x(t-\tau(t)))- g(t-\tau(t)) \|^2] + \frac{\eta\sqrt{t}}{2}\Ex [\|x(t+1)-x(t)\|^2] \\
\leq &\  \frac{2m\sigma^2}{\eta\sqrt{t}} +\frac{\eta\sqrt{t}}{2}\Ex [\|x(t+1)-x(t)\|^2] \nonumber
\end{align}
where we used the fact that $\Ex [\| \nabla f(x(t-\tau(t)))- g(t-\tau(t)) \|^2]\leq m\sigma^2$ for all $t$.
Now applying \eqref{eqn:sumobjdiff_inprod_rhs2_tmp}, \eqref{eqn:Einprod1}
and \eqref{eqn:Einprod2} in \eqref{eqn:Esumobj}, we have
\begin{align}\label{eqn:TEx}
&\qquad T\Ex \left[f(y(T))- f(x^*)\right]  \nonumber \\
&\ \leq \sum_{t=1}^{T} \frac{1}{2\alpha(t)}\left(e(t)-e(t+1)\right)
+ 8\sqrt{2nLT} mRCB + \sum_{t=1}^{T} \frac{2m\sigma^2}{\eta\sqrt{t}} \\
& \ \leq \frac{e(1)}{2\alpha(1)}+\sum_{t=2}^{T} \frac{e(t)}{2}\left(\frac{1}{\alpha(t)}-\frac{1}{\alpha(t-1)}\right) + 8\sqrt{2nLT} mRCB + \sum_{t=1}^{T} \frac{2m\sigma^2}{\eta\sqrt{t}}  \nonumber 
\end{align}
where we note that $\alpha(t)$ is nonincreasing, and hence $\frac{1}{\alpha(t)}-\frac{1}{\alpha(t-1)}\geq0$ and
\begin{equation*}
\sum_{t=2}^{T} \frac{e(t)}{2}\left(\frac{1}{\alpha(t)}-\frac{1}{\alpha(t-1)}\right)
\leq \frac{\DX^2}{2}\sum_{t=2}^{T} \left(\frac{1}{\alpha(t)}-\frac{1}{\alpha(t-1)}\right) 
=\frac{\DX^2}{2}\del{\frac{1}{\alpha(T)}-\frac{1}{\alpha(1)}} 
\end{equation*}
where we used the fact that $e(t)=\Ex[\|x(t)-x^*\|^2]\leq \DX^2:=4mnR^2$ for all $t$.
Plugging this into \eqref{eqn:TEx}, dividing both sides by $T$, 
and using the fact that $\sum_{t=1}^{T}1/\sqrt{t}\leq 2\sqrt{T}$, we obtain
\eqref{eqn:y_opt}.
This completes the proof.
\end{proof}

%

\section*{Acknowledgments}
The authors would like to thank Dr.~WenZhan Song and his SensorWeb Research Laboratory
at University of Georgia for sharing the illustrative Figure \ref{fig:seismic_tomography} 
and 
the three test seismic tomography datasets in this paper. 

\bibliographystyle{abbrv}
\bibliography{library}

\begin{thebibliography}{10}

\bibitem{Agarwal:2011a}
A.~Agarwal and J.~C. Duchi.
\newblock Distributed delayed stochastic optimization.
\newblock In {\em Advances in Neural Information Processing Systems}, pages
  873--881, 2011.

\bibitem{Aysal:2009a}
T.~C. Aysal, M.~E. Yildiz, A.~D. Sarwate, and A.~Scaglione.
\newblock Broadcast gossip algorithms for consensus.
\newblock {\em Signal Processing, IEEE Transactions on}, 57(7):2748--2761,
  2009.

\bibitem{Boyd:2011a}
S.~Boyd, N.~Parikh, E.~Chu, B.~Peleato, and J.~Eckstein.
\newblock Distributed optimization and statistical learning via the alternating
  direction method of multipliers.
\newblock {\em Foundations and Trends{\textregistered} in Machine Learning},
  3(1):1--122, 2011.

\bibitem{Cevher:2014b}
V.~Cevher, S.~Becker, and M.~Schmidt.
\newblock Convex optimization for big data: Scalable, randomized, and parallel
  algorithms for big data analytics.
\newblock {\em Signal Processing Magazine, IEEE}, 31(5):32--43, 2014.

\bibitem{Chang:2016b}
T.-H. Chang, M.~Hong, W.-C. Liao, and X.~Wang.
\newblock Asynchronous distributed admm for large-scale optimization part i:
  Algorithm and linear convergence analysis.
\newblock {\em IEEE Transactions on Signal Processing}, 64(12):3118--3130,
  2016.

\bibitem{Chang:2015a}
T.-H. Chang, M.~Hong, and X.~Wang.
\newblock Multi-agent distributed optimization via inexact consensus admm.
\newblock {\em Signal Processing, IEEE Transactions on}, 63(2):482--497, 2015.

\bibitem{Chang:2016a}
T.-H. Chang, W.-C. Liao, M.~Hong, and X.~Wang.
\newblock Asynchronous distributed admm for large-scale optimization part ii:
  Linear convergence analysis and numerical performance.
\newblock {\em IEEE Transactions on Signal Processing}, 64(12):3131--3144,
  2016.

\bibitem{Duchi:2012b}
J.~C. Duchi, A.~Agarwal, and M.~J. Wainwright.
\newblock Dual averaging for distributed optimization: convergence analysis and
  network scaling.
\newblock {\em Automatic control, IEEE Transactions on}, 57(3):592--606, 2012.

\bibitem{Feyzmahdavian:2014a}
H.~R. Feyzmahdavian, A.~Aytekin, and M.~Johansson.
\newblock A delayed proximal gradient method with linear convergence rate.
\newblock In {\em Machine Learning for Signal Processing (MLSP), 2014 IEEE
  International Workshop on}, pages 1--6. IEEE, 2014.

\bibitem{Forero:2010a}
P.~A. Forero, A.~Cano, and G.~B. Giannakis.
\newblock Consensus-based distributed support vector machines.
\newblock {\em The Journal of Machine Learning Research}, 11:1663--1707, 2010.

\bibitem{Gan:2013a}
L.~Gan, U.~Topcu, and S.~H. Low.
\newblock Optimal decentralized protocol for electric vehicle charging.
\newblock {\em Power Systems, IEEE Transactions on}, 28(2):940--951, 2013.

\bibitem{Iutzeler:2016a}
F.~Iutzeler, P.~Bianchi, P.~Ciblat, and W.~Hachem.
\newblock Explicit convergence rate of a distributed alternating direction
  method of multipliers.
\newblock {\em IEEE Transactions on Automatic Control}, 61(4):892--904, 2016.

\bibitem{Iutzeler:2012a}
F.~Iutzeler, P.~Ciblat, W.~Hachem, and J.~Jakubowicz.
\newblock New broadcast based distributed averaging algorithm over wireless
  sensor networks.
\newblock In {\em Acoustics, Speech and Signal Processing (ICASSP), 2012 IEEE
  International Conference on}, pages 3117--3120. IEEE, 2012.

\bibitem{Jakovetic:2015a}
D.~Jakovetic, J.~M. Moura, and J.~Xavier.
\newblock Linear convergence rate of a class of distributed augmented
  lagrangian algorithms.
\newblock {\em Automatic Control, IEEE Transactions on}, 60(4):922--936, 2015.

\bibitem{Jakovetic:2014a}
D.~Jakovetic, J.~Xavier, and J.~M. Moura.
\newblock Fast distributed gradient methods.
\newblock {\em Automatic Control, IEEE Transactions on}, 59(5):1131--1146,
  2014.

\bibitem{Kraska:2013a}
T.~Kraska, A.~Talwalkar, J.~C. Duchi, R.~Griffith, M.~J. Franklin, and M.~I.
  Jordan.
\newblock Mlbase: A distributed machine-learning system.
\newblock In {\em CIDR}, volume~1, pages 2--1, 2013.

\bibitem{Li:2015a}
J.~Li, G.~Chen, Z.~Dong, and Z.~Wu.
\newblock Distributed mirror descent method for multi-agent optimization with
  delay.
\newblock {\em Neurocomputing}, 2015.

\bibitem{Li:2013b}
M.~Li, D.~G. Andersen, and A.~Smola.
\newblock Distributed delayed proximal gradient methods.
\newblock In {\em NIPS Workshop on Optimization for Machine Learning}, 2013.

\bibitem{Li:2014a}
M.~Li, D.~G. Andersen, A.~J. Smola, and K.~Yu.
\newblock Communication efficient distributed machine learning with the
  parameter server.
\newblock In {\em Advances in Neural Information Processing Systems}, pages
  19--27, 2014.

\bibitem{Liu:2015a}
J.~Liu, S.~J. Wright, C.~R{\'e}, V.~Bittorf, and S.~Sridhar.
\newblock An asynchronous parallel stochastic coordinate descent algorithm.
\newblock {\em The Journal of Machine Learning Research}, 16(1):285--322, 2015.

\bibitem{Lo:2013a}
C.-H. Lo and N.~Ansari.
\newblock Decentralized controls and communications for autonomous distribution
  networks in smart grid.
\newblock {\em Smart Grid, IEEE Transactions on}, 4(1):66--77, 2013.

\bibitem{Lovasz:1993a}
L.~Lov{\'a}sz.
\newblock Random walks on graphs: A survey.
\newblock {\em Combinatorics, Paul Erd\H{o}s is eighty}, 2(1):1--46, 1993.

\bibitem{Makhdoumi:2017a}
A.~Makhdoumi and A.~Ozdaglar.
\newblock Convergence rate of distributed admm over networks.
\newblock {\em IEEE Transactions on Automatic Control}, 2017.

\bibitem{Mokhtari:2015b}
A.~Mokhtari and A.~Ribeiro.
\newblock Decentralized double stochastic averaging gradient.
\newblock In {\em Signals, Systems and Computers, 2015 49th Asilomar Conference
  on}, pages 406--410. IEEE, 2015.

\bibitem{Nedic:2015a}
A.~Nedic and A.~Olshevsky.
\newblock Distributed optimization over time-varying directed graphs.
\newblock {\em Automatic Control, IEEE Transactions on}, 60(3):601--615, 2015.

\bibitem{Nedic:2016a}
A.~Nedi{\'c} and A.~Olshevsky.
\newblock Stochastic gradient-push for strongly convex functions on
  time-varying directed graphs.
\newblock {\em IEEE Transactions on Automatic Control}, 61(12):3936--3947,
  2016.

\bibitem{Nedic:2009b}
A.~Nedic and A.~Ozdaglar.
\newblock Distributed subgradient methods for multi-agent optimization.
\newblock {\em Automatic Control, IEEE Transactions on}, 54(1):48--61, 2009.

\bibitem{Nesterov:1983a}
Y.~Nesterov.
\newblock A method for unconstrained convex minimization problem with the rate
  of convergence o (1/k2).
\newblock Technical Report~3, Doklady AN SSSR, 1983.

\bibitem{Olfati-Saber:2004a}
R.~Olfati-Saber and R.~M. Murray.
\newblock Consensus problems in networks of agents with switching topology and
  time-delays.
\newblock {\em Automatic Control, IEEE Transactions on}, 49(9):1520--1533,
  2004.

\bibitem{Rabbat:2004a}
M.~Rabbat and R.~Nowak.
\newblock Distributed optimization in sensor networks.
\newblock In {\em Proceedings of the 3rd international symposium on Information
  processing in sensor networks}, pages 20--27. ACM, 2004.

\bibitem{Sayed:2014a}
A.~Sayed.
\newblock Adaptation, learning, and optimization over networks.
\newblock {\em Foundations and Trends{\textregistered} in Machine Learning},
  7(4-5):311--801, 2014.

\bibitem{Sayed:2013b}
A.~H. Sayed, S.-Y. Tu, and J.~Chen.
\newblock Online learning and adaptation over networks: More information is not
  necessarily better.
\newblock In {\em Information Theory and Applications Workshop (ITA), 2013},
  pages 1--8. IEEE, 2013.

\bibitem{Shamir:2014a}
O.~Shamir and N.~Srebro.
\newblock Distributed stochastic optimization and learning.
\newblock In {\em Communication, Control, and Computing (Allerton), 2014 52nd
  Annual Allerton Conference on}, pages 850--857. IEEE, 2014.

\bibitem{Shi:2015a}
W.~Shi, Q.~Ling, G.~Wu, and W.~Yin.
\newblock Extra: An exact first-order algorithm for decentralized consensus
  optimization.
\newblock {\em SIAM Journal on Optimization}, 25(2):944--966, 2015.

\bibitem{Shi:2014a}
W.~Shi, Q.~Ling, K.~Yuan, G.~Wu, and W.~Yin.
\newblock On the linear convergence of the admm in decentralized consensus
  optimization.
\newblock {\em Signal Processing, IEEE Transactions on}, 62(7):1750--1761,
  2014.

\bibitem{Song:2009a}
W.-Z. Song, R.~Huang, M.~Xu, A.~Ma, B.~Shirazi, and R.~LaHusen.
\newblock Air-dropped sensor network for real-time high-fidelity volcano
  monitoring.
\newblock In {\em Proceedings of the 7th international conference on Mobile
  systems, applications, and services}, pages 305--318. ACM, 2009.

\bibitem{Sra:2016a}
S.~Sra, A.~W. Yu, M.~Li, and A.~J. Smola.
\newblock Adadelay: Delay adaptive distributed stochastic convex optimization.
\newblock In {\em Proceedings of the 19th International Conference on
  Artificial Intelligence and Statistics}, volume~51, pages 957--965, 2016.

\bibitem{Tian:2008a}
Y.-P. Tian and C.-L. Liu.
\newblock Consensus of multi-agent systems with diverse input and communication
  delays.
\newblock {\em Automatic Control, IEEE Transactions on}, 53(9):2122--2128,
  2008.

\bibitem{Tseng:2008a}
P.~Tseng.
\newblock On accelerated proximal gradient methods for convex-concave
  optimization.
\newblock {\em submitted to SIAM Journal on Optimization}, 2008.

\bibitem{Tsitsiklis:1984a}
J.~N. Tsitsiklis.
\newblock Problems in decentralized decision making and computation.
\newblock Technical report, DTIC Document, 1984.

\bibitem{Wang:2015a}
H.~Wang, X.~Liao, T.~Huang, and C.~Li.
\newblock Cooperative distributed optimization in multiagent networks with
  delays.
\newblock {\em Systems, Man, and Cybernetics: Systems, IEEE Transactions on},
  45(2):363--369, 2015.

\bibitem{Wei:2013b}
E.~Wei and A.~Ozdaglar.
\newblock On the o(1/k) convergence of asynchronous distributed alternating
  direction method of multipliers.
\newblock In {\em Global Conference on Signal and Information Processing
  (GlobalSIP), 2013 IEEE}, pages 551--554. IEEE, 2013.

\bibitem{Wu:2016a}
T.~Wu, K.~Yuan, Q.~Ling, W.~Yin, and A.~H. Sayed.
\newblock Decentralized consensus optimization with asynchrony and delays.
\newblock In {\em Proceedings of IEEE Asilomar Conference on Signals, Systems,
  and Computers}, 2016.

\bibitem{Xiao:2004a}
L.~Xiao and S.~Boyd.
\newblock Fast linear iterations for distributed averaging.
\newblock {\em Systems \& Control Letters}, 53(1):65--78, 2004.

\bibitem{Yuan:2015a}
D.~Yuan, D.~W. Ho, and S.~Xu.
\newblock Regularized primal-dual subgradient method for distributed
  constrained optimization.
\newblock {\em IEEE Transactions on Cybernetics}, 2015.

\bibitem{Yuan:2016a}
K.~Yuan, Q.~Ling, and W.~Yin.
\newblock On the convergence of decentralized gradient descent.
\newblock {\em SIAM Journal on Optimization}, 26(3):1835--1854, 2016.

\bibitem{Zhang:2014b}
R.~Zhang and J.~Kwok.
\newblock Asynchronous distributed admm for consensus optimization.
\newblock In {\em Proceedings of the 31st International Conference on Machine
  Learning (ICML-14)}, pages 1701--1709, 2014.

\bibitem{Zhang:2016a}
W.~Zhang, S.~Gupta, X.~Lian, and J.~Liu.
\newblock Staleness-aware async-sgd for distributed deep learning.
\newblock In {\em Proceedings of the 25th International Joint Conference on
  Artificial Intelligence}, pages 2350--2356, 2016.

\bibitem{Zhao:2015b}
L.~Zhao, W.-Z. Song, L.~Shi, and X.~Ye.
\newblock Decentralised seismic tomography computing in cyber-physical sensor
  systems.
\newblock {\em Cyber-Physical Systems}, pages 1--22, 2015.

\bibitem{Zhao:2015a}
L.~Zhao, W.-Z. Song, and X.~Ye.
\newblock Fast decentralized gradient descent method and applications to
  in-situ seismic tomography.
\newblock In {\em Big Data (Big Data), 2015 IEEE International Conference on},
  pages 908--917. IEEE, 2015.

\end{thebibliography}

\end{document}